\documentclass[11pt]{article}
\usepackage{amsmath,array,amssymb,latexsym}
\usepackage{amsthm} 
\newtheorem{theorem}{Theorem}[section]
\newtheorem{proposition}[theorem]{Proposition} 
\newtheorem{corollary}[theorem]{Corollary}
\newtheorem{lemma}[theorem]{Lemma}
\newtheorem{claim}[theorem]{Claim}
\newtheorem{definition}[theorem]{Definition}
\newtheorem{remark}[theorem]{Remark}
\numberwithin{equation}{section}
\begin{document}
\title{Convergence of K\"ahler to real polarizations \\ on flag manifolds via toric degenerations}\author{Mark D. Hamilton \and Hiroshi Konno}
\maketitle 
\begin{abstract}
In this paper we construct a family of complex structures on a complex flag manifold that converge to the real polarization coming from the Gelfand-Cetlin integrable system, in the sense that holomorphic sections of a prequantum line bundle converge to delta-function sections supported on the Bohr-Sommerfeld fibers.  
Our construction is based on a toric degeneration of flag varieties and a deformation of K\"ahler structure on toric varieties by symplectic potentials.
\end{abstract}

\section{Introduction}
Let $(M, \omega)$ be a $2n$-dimensional symplectic manifold. 
A prequantum line bundle $(L, h, \nabla)$ is a complex line bundle $L$ on $M$ with a Hermitian metric $h$ and a Hermitian connection $\nabla$, whose curvature equals $-2\pi\sqrt{-1}\, \omega$. 
Geometric quantization is a procedure to assign a certain vector space, which is called a quantum Hilbert space, to $(M, \omega)$.
To perform a geometric quantization procedure, we must choose a polarization, which is an integrable Lagrangian subbundle of the (complexification of the) tangent bundle $TM$ of $M$.  
Then the quantum Hilbert space $\mathcal{H}(P)$ for a polarization $P$ is naively a subspace  of (a certain completion of) the space of sections of $L$, consisting of covariantly constant sections along the polarization $P$.

The most common example of a polarization comes from an integrable complex structure $J$ on $M$ such that $(M, \omega, J)$ is a K\"ahler manifold.  In this case the anti-holomorphic tangent bundle $P_J=T^{0,1}M$ is a polarization, which we call a K\"ahler polarization. 
The quantum Hilbert space $\mathcal{H}(P_J)$ is the space of holomorphic sections $H^0(L,\overline{\partial}^J)$ with respect to the natural holomorphic structure  $\overline{\partial}^J$ on $L$ induced by $J$. 
 
Another type of polarization, called a real polarization, is given by a foliation of $M$ into Lagrangian submanifolds.  
A completely integrable system $\mu \colon M \to \mathbb{R}^n$ (which is assumed to be proper) defines a singular real polarization $P_\mu$, where $(P_\mu)_x$ is the tangent space of the fiber of $\mu$ at each point $x \in M$. 
We set $BS(\mu)=\{ p \in \mu(M)~|~ H^0((L,h,\nabla)|_{\mu^{-1}(p)})\ne 0 \}$, where $H^0((L,h,\nabla)|_{\mu^{-1}(p)})=\{ s \in \Gamma((L, h,\nabla)|_{\mu^{-1}(p)})~|~ \nabla s=0 \}$.
Namely, $p \in BS(\mu)$ if and only if $\mu^{-1}(p)$ is a Bohr-Sommerfeld fiber. 
Then the quantum Hilbert space $\mathcal{H}(P_\mu)$ is defined to be $\bigoplus_{p \in BS(\mu)}H^0((L,h,\nabla)|_{\mu^{-1}(p)}) $ \cite{S}. 

From the point of view of physics, the quantum Hilbert space should be independent of the choice of polarization. 
In particular, although K\"ahler and real polarizations seem to be quite different, the quantum Hilbert space for a K\"ahler polarization should be isomorphic to the one for a real polarization. 
There are several examples where this principle is observed to be true. 
A non-singular projective toric variety has a natural K\"ahler structure, and its moment map for the torus action induces a (singular) real polarization. 
It is well known that the dimension of the space of holomorphic sections of the prequantum line bundle is the number of lattice points in the image of the moment map, which is also the number of Bohr-Sommerfeld fibers in the variety.
This implies that the principle holds in this case.
In \cite{JW} Jeffrey-Weitsman showed that the principle also holds in the case of the moduli space of flat connections over a compact Riemann surface.

A flag manifold with an integral symplectic structure has a singular real polarization defined by the Gelfand-Cetlin system, which was introduced by Guillemin-Sternberg in \cite{GS}, as well as a natural K\"ahler polarization 
since it is a complex manifold.  
In \cite{GS} the authors studied the quantization of flag manifolds, and showed that the two polarizations give rise to quantizations with the same dimensions.
However, their proof did not give any sort of direct relationship between the quantizations; rather, they computed the dimensions of the quantizations by other means (representation-theoretical and combinatoric) and showed they are equal.

One way of approaching the principle of independence of polarization is the following, considered by Baier, Florentino, Mour\~ao and Nunes in \cite{BFMN}.
Fix a K\"ahler polarization $P_J$ and a real polarization $P_\mu$ on $(M,\omega)$ respectively. 
Then the principle can be understood naturally if there is a family $\{ P_{J_s} \}_{s \in [0, \infty)}$ of K\"ahler polarizations on $M$ with $P_{J_0}=P_J$ which converges to $P_\mu$ in the sense that there exists a basis $\{ \sigma^m_s \}_{m \in BS(P_\mu)}$ of $\mathcal{H}(P_{J_s})$ for each $s \in [0, \infty)$ such that, for each $m \in BS(P_\mu)$, $\sigma^m_s$ converges to a delta-function section supported on the Bohr-Sommerfeld fiber $\mu^{-1}(m)$ as $s$ goes to $\infty$. 
In \cite{BFMN}, the authors carried out such a construction in the case of a non-singular projective toric variety by changing symplectic potentials, an important notion in the deformation theory of toric K\"ahler structures due to Guillemin \cite{Gu1,Gu2} and Abreu \cite{Ab1,Ab2}.

In this paper we construct a family of K\"ahler polarizations on a flag manifold that converge to the real polarization coming from the Gelfand-Cetlin system.  
See Theorem \ref{mainth} for details. 
In doing so, we provide a direct relationship between the two quantizations.
Our construction is based on the construction due to \cite{BFMN} and the toric degeneration of a flag variety due to Kogan and Miller \cite{KM}. 
Originally, a toric degeneration of a flag variety was constructed in terms of representation theory \cite{GL,C}.
Later Kogan and Miller introduced deformed actions of a Borel subgroup on the space of matrices and described a toric degeneration of a flag variety explicitly.  
Moreover, they constructed a ``degeneration in stages'' of a flag variety to study the geometric meaning of the Gelfand-Cetlin basis of the irreducible representation of the unitary group. 
In \cite{NNU} Nishinou, Nohara and Ueda pointed out that through the degeneration in stages one can identify the Gelfand-Cetlin system on the flag manifold with the integrable system on the limiting toric variety.

Our construction of a family of K\"ahler polarizations on a flag manifold proceeds as follows. 
We start from a flag manifold $Fl_n$ embedded in the product of projective spaces $\mathbb{P}=\prod_{l=1, \dots,n-1}\mathbb{P}(\bigwedge^l \mathbb{C}^n)$. 
For each $(a_1, \dots, a_{n-1}) \in (\mathbb{Z}_{>0})^{n-1}$ we fix a prequantum line bundle on $\mathbb{P}$ inducing a natural symplectic structure on $Fl_n$.  
The toric degeneration of the flag variety $Fl_n$ due to \cite{KM} is a family of complex subvarieties $\{ V_t \}_{t \in \mathbb{C}}$ in $\mathbb{P}$, where $V_1= Fl_n$ and $V_0$ is a toric variety. 
Since all $V_t$ are diffeomorphic to each other for $t \ne 0$, the family $\{ V_t \}_{t \ne 0}$ can be considered as a family of K\"ahler structures on the flag manifold $Fl_n$.
On the other hand, there is a family of toric K\"ahler structures $\{V_{0,s} \}_{s \in [0, \infty)}$ on $V_0$ with $V_{0,0}=V_0$, as considered in \cite{BFMN} (explained above).
If we could identify $Fl_n$ with $V_{0, s}$ as a symplectic manifold, we could pull back the complex structures on $V_{0,s}$ to $Fl_n$.
However, the toric variety $V_0$ is not diffeomorphic to the flag manifold $Fl_n$. 

Instead, we consider a space $V_t$ that is an approximation to $V_0$ that is still diffeomorphic to the flag manifold $Fl_n$. 
We show that the deformation $\{V_{0,s} \}_{s \in [0, \infty)}$ can be realized as the restriction of a deformation of the ambient toric variety $\mathbb{P}$. 
The deformation of the ambient space induces a family of K\"ahler structures $\{ V_{t,s} \}_{s \in [0,\infty)}$ on $V_t$ with $V_{t,0}=V_t$ for each $t \in \mathbb{C}$. 
We develop a method to identify $V_{t,s}$ with $V_{t,0}=V_t$ as a symplectic manifold. 
Moreover, we identify $Fl_n$ with $V_t$ as a symplectic manifold by using the gradient-Hamiltonian flow (a notion that is due to Ruan~\cite{R}) along a path that is an approximation of the path for degeneration in stages. 
Hence we can pull the complex structure of $V_{t,s}$ back to $Fl_n$. 
We also lift this identification to the prequantum line bundle in order to pull back holomorphic sections. 
Thus we have a family of complex structures on the flag manifold with a fixed symplectic structure 
and a family of sections of the prequantum line bundle on the flag manifold, which are holomorphic with respect to the corresponding complex structure. 
Moreover, we give a precise estimate of these holomorphic sections, which allows us to prove that the holomorphic sections converge to delta-function sections supported on the Bohr-Sommerfeld fibers if we perform these two types of deformations simultaneously in an appropriate way.

The content of this paper is organized as follows. 
In Section 2 we state our main result. 
We review the results on a toric degeneration of a flag variety in Section 3. 
Then we recall the gradient-Hamiltonian flow and construct its lift to the line bundle in Section 4. 
In Section 5 we review toric K\"ahler structures of toric manifolds, in particular, their deformation due to \cite{BFMN}. 
In Section 6 we develop a method to identify submanifolds under the deformation of toric K\"ahler structures of the ambient toric manifolds. 
We also give an estimate of the change of holomorphic sections under this deformation. 
In Section 7 we prove the main result, constructing a family of complex structures on the flag manifold, and proving that holomorphic sections converge to delta-function sections supported on Bohr-Sommerfeld fibers. 

The first named author is supported by the JSPS Postdoctoral Fellowship for Foreign Researchers.
The second named author is supported by the JSPS Grant-in-Aid for Scientific Research (C), No. 19540067. 
The authors would like to thank Y.Nohara for discussions. 
\section{Main results}\label{main}
Let $GL_n$ and $B$ be the general linear group and its Borel subgroup consisting of upper triangular matrices with $\mathbb{C}$-coefficient respectively. 
The flag manifold is defined to be a complex manifold $Fl_n = GL_n/B$. 
Let $\Lambda_n$ be the set of increasing indexes $I=(i_1 < \dots < i_l)$ with $1 \le i_1, i_l \le n$.  
For $I=(i_1 < \dots < i_l) \in \Lambda_n$ and $V=(v_{ij}) \in GL_n$ we set $|I|=l$ and
\begin{align*}
P_I(V)= \det \begin{pmatrix}v_{i_1 1} & \ldots & v_{i_1 l} \\
                            \vdots & \ddots & \vdots \\
                            v_{i_l 1} & \ldots & v_{i_l l}\end{pmatrix}.
\end{align*}
Then the Plu\"cker embedding 
$$
\rho \colon Fl_n \to \mathbb{P}=\prod_{l=1}^{n-1} \mathbb{P}(\bigwedge^l \mathbb{C}^n)
$$
is defined by $[V] \mapsto ([p_I(V); |I|=1], \dots, [p_I(V); |I|=n-1])$,  
where $[x_I; |I|=l]$ is the homogeneous coordinate of $\mathbb{P}(\bigwedge^l \mathbb{C}^n)$.
Since the left $U(n)$-action on $M_n(\mathbb{C})$ commutes with the right $B$-action on $M_n(\mathbb{C})$, $U(n)$ acts on $Fl_n$ from the left.

Next we define a holomorphic line bundle on $Fl_n$ and a Hermitian connection on it. 
Let $H_l$ be the hyperplane bundle on $\mathbb{P}(\bigwedge^l \mathbb{C}^n)$. 
It has a natural Hermitian metric $h_l$ such that $\frac{\sqrt{-1}}{2 \pi}R^{\nabla^l} = \omega_l$, where $R^{\nabla^l}$ is the curvature of the Chern connection $\nabla_l$ for the Hermitian metric $h_l$, and $\omega_l \in \Omega^2(\mathbb{P}(\bigwedge^l \mathbb{C}^n))$ is the Fubini-Study form. 
Let $\pi_l \colon \mathbb{P} \to \mathbb{P}(\bigwedge^l \mathbb{C}^n)$ be the projection.
Fix $\mathbf{a}=(a_1, \dots, a_{n-1}) \in (\mathbb{Z}_{>0})^{n-1}$. 
Then we define a K\"ahler form $\omega_{\mathbb{P}}$ and a prequantum line bundle $(L^{\mathbb{P}}, h^{\mathbb{P}}, \nabla^{\mathbb{P}})$ on $\mathbb{P}$ by
$$
\omega_{\mathbb{P}}= \sum_{l=1}^{n-1}a_l \pi_l^* \omega_l \in \Omega^2(\mathbb{P}) , ~~(L^{\mathbb{P}}, h^{\mathbb{P}}, \nabla^{\mathbb{P}})= \bigotimes_{l=1}^{n-1} \pi_l^* (H_l, h_l, \nabla^l)^{ \otimes a_l}.
$$
Then $\nabla^{\mathbb{P}}$ is the Chern connection of $(L^{\mathbb{P}}, h^{\mathbb{P}})$ and satisfies $\frac{\sqrt{-1}}{2 \pi}R^{\nabla^{\mathbb{P}}} = \omega_{\mathbb{P}}$.    
We set $(L^{Fl_n}, h^{Fl_n}, \nabla^{Fl_n})=\rho^*(L^{\mathbb{P}}, h^{\mathbb{P}}, \nabla^{\mathbb{P}})$, that is, $L^{Fl_n}$ is a holomorphic line bundle on $Fl_n$ with a Hermitian metric $h^{Fl_n}$ and the Chern connection $\nabla^{Fl_n}$ whose first Chern form is $\rho^*\omega_{\mathbb{P}}$.
The $U(n)$-action on $Fl_n$ preserves $\rho^*\omega_{\mathbb{P}}$ with a moment map $\mu_{U(n)} \colon Fl_n \to \mathfrak{u}(n)^*$. 

Next we recall a certain completely integrable system on $Fl_n$. 
Consider $U(l)$ for $l=1, \dots, n-1$ as a subgroup of $U(n)$ defined by $U(l)= \{ \begin{pmatrix} A & O_{l,n-l} \\ O_{n-l,l} & E_{n-l} \end{pmatrix}\! \in \! U(n) \}$, where $O_{l,n-l} \! \in \! M_{l,n-l}(\mathbb{C})$ and $O_{n-l,l} \! \in \! M_{n-l,l}(\mathbb{C})$ are the zero matrices, $E_{n-l} \in M_{n-l}(\mathbb{C})$ is the unit element, and $A \in M_l(\mathbb{C})$. 
Let $\iota_l^* \colon \mathfrak{u}(n)^* \to \mathfrak{u}(l)^*$ be the dual map of the inclusion $\iota_l \colon \mathfrak{u}(l) \to \mathfrak{u}(n)$.
Define a map $\lambda_l^j \colon \mathfrak{u}(l) \to \mathbb{R}$ such that $\lambda_l^1(A) \ge \dots \ge \lambda_l^l(A)$ are eigenvalues of $-\sqrt{-1}A$ for $A \in \mathfrak{u}(l)$. 
We identify $\mathfrak{u}(l)$ with $\mathfrak{u}(l)^*$ by the invariant inner product. 
In \cite{GS} Guillemin and Sternberg proved that 
$$
\mu_{GC}=(\lambda_l^j \circ \iota_l^* \circ \mu_{U(n)}~;~ 1 \le l \le n-1,~ 1 \le j \le l) \colon Fl_n \to \mathbb{R}^d
$$
is a completely integrable system, where $d= \frac{1}{2} \dim_{\mathbb{R}}Fl_n = \frac{n(n-1)}{2}$. 
The completely integrable system $\mu_{GC} \colon Fl_n \to \mathbb{R}^d$ and its image $\Delta_{GC}=\mu_{GC}(Fl_n) \subset \mathbb{R}^d$ are called the Gelfand-Cetlin system and the Gelfand-Cetlin polytope respectively. 
Note that $\mu_{GC} \colon Fl_n \to \mathbb{R}^d$ is a continuous map and that it is smooth on $\mu_{GC}^{-1}(\mathrm{Int}\Delta_{GC})$, where $\mathrm{Int}\Delta_{GC}$ is the interior of $\Delta_{GC}$.
Moreover, $\mu_{GC}^{-1}(m)$ is a $d$-dimensional real torus for each $m \in \mathrm{Int}\Delta_{GC}$.
In \cite{GS} Guillemin and Sternberg also proved that, for $m \in \mathrm{Int}\Delta_{GC} \subset \mathbb{R}^d$, the fiber $\mu_{GC}^{-1}(m)$ is a Bohr-Sommerfeld fiber if and only if $m \in \mathrm{Int}\Delta_{GC} \cap \mathbb{Z}^d$
and that the number of the points $\Delta_{GC} \cap \mathbb{Z}^d$ coincides with the dimension of the space of holomorphic sections $H^0(L^{Fl_n}, \overline{\partial}^{Fl_n})$, where $\overline{\partial}^{Fl_n}$ is the holomorphic structure on $L^{Fl_n}$.
Namely the quantum Hilbert space for the K\"ahler polarization on $Fl_n$ is isomorphic to the one for the real polarization $P_{\mu_{GC}}$ coming from the Gelfand-Cetlin system $\mu_{GC}$.

In this paper we construct a family of complex structures $\{ J_s \}_{s \in [0, \infty)}$ on $Fl_n$ such that the family of K\"ahler polarizations $\{ P_{J_s} \}_{s \in [0, \infty)}$ converge to the real polarization $P_{\mu_{GC}}$ in the following sense. 
\begin{theorem}\label{mainth}
Let $\mathbb{F}$ and $J_{\mathbb{F}}$ be the underlying $C^\infty$-manifold and the complex structure of $Fl_n$ respectively. 
Set $\omega_{\mathbb{F}}=\rho^* \omega_{\mathbb{P}} \in \Omega^2(\mathbb{F})$ and $d= \dim_{\mathbb{R}}\mathbb{F} =\frac{n(n-1)}{2}$.
Let $(L^{\mathbb{F}},h^{\mathbb{F}},\nabla^{\mathbb{F}})$ be the underlying $C^\infty$ line bundle of $(L^{Fl_n}, h^{Fl_n}, \nabla^{Fl_n})$.  
Then there exists a one parameter family of $\{ J_s \}_{s \in [0, \infty)}$ of complex structures on $\mathbb{F}$ which satisfies the following:
\newline
$(1)$ $J_s$ is continuous with respect to the parameter $s \in [0,\infty)$.
\newline
$(2)$ $J_0=J_{\mathbb{F}}$
\newline
$(3)$ $(\mathbb{F},\omega_{\mathbb{F}},J_s)$ is a K\"ahler manifold for each $s \in [0,\infty)$. 
So, for each $s \in [0,\infty)$, the Hermitian line bundle $(L^{\mathbb{F}},h^{\mathbb{F}},\nabla^{\mathbb{F}})$ induces the holomorphic structure $\overline{\partial}^s$ on $L^{\mathbb{F}}$. 
\newline
$(4)$ For each $s \in [0,\infty)$ there exists a basis $\{ \sigma^m_s ~|~ m \in \Delta_{GC} \cap \mathbb{Z}^d \}$ of the space of holomorphic sections $H^0(L^{\mathbb{F}}, \overline{\partial}^s)$ such that, for each $m \in \mathrm{Int}\Delta_{GC} \cap \mathbb{Z}^d$, the section $\frac{\sigma^m_s}{\Vert \sigma^m_s \Vert_{L^1(\mathbb{F})}}$ converges to a delta-function section supported on the Bohr-Sommerfeld fiber $\mu_{GC}^{-1}(m) $ in the following sense:
there exist a covariantly constant section $\delta_m^{\mathbb{F}}$ of $(L^{\mathbb{F}}, h^{\mathbb{F}}, \nabla^{\mathbb{F}})|_{\mu_{GC}^{-1}(m)}$ and a measure $d \theta_m$ on $\mu_{GC}^{-1}(m)$ such that, for any smooth section $\phi$ of the dual line bundle $(L^{\mathbb{F}})^*$, the following holds
$$
\lim_{s \to \infty} \int_{\mathbb{F}} \biggl\langle \phi, \frac{\sigma^m_s}{\Vert \sigma^m_s \Vert_{L^1(\mathbb{F})}} \biggr\rangle \frac{\omega_{\mathbb{F}}^d}{d !} = \int_{\mu_{GC}^{-1}(m)} \langle \phi , \delta_m^{\mathbb{F}} \rangle d \theta_m.
$$
\end{theorem}
\begin{remark} 
By a similar argument as in the proof of Theorem \ref{mainth} we can also prove that the support of the section $\sigma^m_s$ converges to $\mu_{CG}^{-1}(m)$ as $s \to \infty$ for any $m \in (\Delta_{GC} \setminus \mathrm{Int}\Delta_{GC}) \cap \mathbb{Z}^d$. 
However, we cannot prove that $\sigma^m_s$ converges to a delta-function section for $m \in (\Delta_{GC} \setminus \mathrm{Int}\Delta_{GC}) \cap \mathbb{Z}^d$, because we do not yet have a sufficient description of $\mu_{CG}^{-1}(m)$. 
\end{remark}
\section{Toric degeneration of flag varieties}
In \cite{KM}  Kogan and Miller constructed  a toric degeneration of a flag variety based on a deformed Borel action. 
They also introduced degeneration in stages of a flag variety. 
In this section we review their construction and recall its symplectic geometric aspects due to Nishinou, Nohara and Ueda \cite{NNU}.

\subsection{Deformed Borel action and toric degeneration}\label{tg}
First we define the right action $\bullet $ of the product group $(GL_n)^n$ on $M_n(\mathbb{C})$ by 
$$
V \bullet g= \begin{pmatrix} \mathbf{v}_1 g_1 \\
                             \vdots \\
                          \mathbf{v}_n g_n   \end{pmatrix} 
~\text{for $V \! = \! \begin{pmatrix} \mathbf{v}_1 \\
                             \vdots \\
                          \mathbf{v}_n   \end{pmatrix} \! \in \! M_n(\mathbb{C}) $ and $g \! = \! (g_1, \dots, g_n) \! \in \! (GL_n)^n$.}
$$
Set $M_n(\mathbb{C}^\times)= \{ (a_{ij}) \in M_n(\mathbb{C})~|~ a_{ij} \ne 0 ~\text{for $i,j=1, \dots, n$} \}$. 
Define a map $\iota \colon M_n(\mathbb{C}^\times) \to (GL_n)^n$ by 
$$
\iota((a_{ij}))=\left( 
\left(\begin{smallmatrix} a_{11} & & O \\
                  & \ddots & \\
                  O &  & a_{1n} \end{smallmatrix}\right), \dots, 
\left(\begin{smallmatrix} a_{n1} & & O \\
                                    & \ddots & \\
                                   O &  & a_{nn} \end{smallmatrix}\right)\right).
$$
Note that $\iota(M_n(\mathbb{C}^\times))$ is the maximal torus of $(GL_n)^n$.  
We also set 
$$
T_{GC}= \left\lbrace \left. \iota \left( \left(\begin{smallmatrix} 1 & & & 1 \\
              a_{21} &   & \ddots &  \\
         \vdots  & \ddots & &                  \\
                  a_{n1} & \dots &  a_{nn\!-\!1} & 1 \end{smallmatrix}\right)\right) ~ \right\vert \left(\begin{smallmatrix} 1 & & & 1 \\
              a_{21} &   & \ddots &  \\
         \vdots  & \ddots & &                  \\
                  a_{n1} & \dots &  a_{nn\!-\!1} & 1 \end{smallmatrix}\right) \in M_n(\mathbb{C}^\times)  \right\rbrace .
$$
We also define a $k$-dimensional algebraic subtorus $T_{GC}^{(k)}$ of $T_{GC}$ by
$$T_{GC}^{(k)}= \{ \iota((a_{ij})) ~|~ (a_{ij}) \in M_n(\mathbb{C}^\times), ~~\text{$i =k\!+\!1$ and $j \le k$ if $a_{ij}\ne 1$} \}.$$
Then we have 
$$
T_{GC}= \{ 1 \} \times  T_{GC}^{(1)} \times \dots \times T_{GC}^{(n-2)} \times T_{GC}^{(n-1)}.
$$ 

Next we define the deformed Borel action as follows. 
For $t \in \mathbb{C}^\times$ we define $t^\omega \in M_n(\mathbb{C}^\times)$ by
\begin{equation}\label{omega}
(t^\omega)_{ij}=t^{\omega_{ij}},  ~~\text{ where $\omega_{ij}= \begin{cases}3^{i-j-1} & \text{if $i>j$,} \\ 0 & \text{if $i \le j$.}\end{cases}$}
\end{equation}
In the above $(t^\omega)_{ij}$ is the $(i,j)$-component of $t^\omega \in M_n(\mathbb{C}^\times)$.
Then we define the deformed action $\bullet_t$ of $B$ on $M_n(\mathbb{C})$ by 
$$
V \bullet_t b = V \bullet \{ \iota(t^\omega) (b,\dots,b) (\iota(t^\omega))^{-1} \}, 
$$
where $\iota(t^\omega), (b,\dots,b), \iota(t^\omega)^{-1} \in (GL_n)^n$.

Let $\mathbb{C}[v_{ij}~|~1 \le i,j \le n]$ be the coordinate ring of $M_n(\mathbb{C})$. 
Let $U \subset B$ the subgroup consisting of the matrices with 1's on the diagonals. 
Then the ring of $U$-invariant functions $\mathbb{C}[v_{ij}~|~1 \le i,j \le n]^U$ for the deformed action $\bullet_t$ of $U$ is generated by the deformed Plu\"cker coordinates
$$
\{ q_I(V,t)= d_I(t^\omega)^{-1}p_I(V \bullet \iota(t^\omega))~|~ I \in \Lambda_n \},  ~~~\text{where $d_I(t^\omega)= \prod_{k=1}^{|I|} (t^\omega)_{i_k k}.$}
$$
From the definition of $\omega \in M_n(\mathbb{Z})$ we see that $q_I(V,t)$ is a polynomial of $v_{ij}~(1 \le i,j \le n)$ and $t$. 
Moreover, the deformed action $\bullet_t$ can be naturally extended to the case $t=0$. 
Thus we have a quotient $Fl_n(t)= M_n(\mathbb{C})/\!/_t B$ for all $t \in \mathbb{C}$, where the right hand side is a GIT quotient by the deformed action $\bullet_t$. 
We also have a family $f \colon (M_n(\mathbb{C}) \times \mathbb{C})/\!/B \to \mathbb{C}$ with $f^{-1}(t)=Fl_n(t)$. 
$Fl_n(1)$ is nothing but the flag variety $Fl_n$.
Note that each $Fl_n(t)$ is embedded in $\mathbb{P}$ by the deformed Plu\"cker embedding $\rho_t \colon Fl_n(t) \to \mathbb{P}$, which is defined by $[V] \mapsto ([q_I(V,t); |I|=1], \dots, [q_I(V,t); |I|=n-1])$ as in the case of the usual Plu\"cker embedding. 
In \cite{KM} Kogan and Miller proved the following, based on the argument in \cite{GL}.
\begin{proposition}\label{propkm} 
$(1)$ 
The family $f \colon (M_n(\mathbb{C}) \times \mathbb{C})/\!/B \to \mathbb{C}$ is flat. 
\newline
$(2)$ $Fl_n(t)$ is biholomorphic to $Fl_n$ for any $t \in \mathbb{C}^\times$. 
Moreover, $Fl_n(0)$ is a toric variety on which the torus $T_{GC}$ acts with an open dense orbit. 

\end{proposition}
Let us give a few remarks about Proposition \ref{propkm}.
Note that, if we set 
$$
GL_n(t)= \{ V \in M_n(\mathbb{C})~|~ V \bullet \iota(t^\omega) \in GL_n \}, 
$$
then we have $Fl_n(t)=GL_n(t)/_t B$ for $t \in \mathbb{C}^\times$, where the right hand side is a geometric quotient of $GL_n(t)$ by the deformed action $\bullet_t$ of the Borel subgroup $B$. 
So we see that $Fl_n(t)$ is biholomorphic to $Fl_n$ for any $t \in \mathbb{C}^\times$.
Moreover, since the action $\bullet g$ on $M_n(\mathbb{C})$ for $g \in T_{GC}$ commutes with the action $\bullet_0 b$ on $M_n(\mathbb{C})$ for $b \in B$, the torus $T_{GC}$ acts on $Fl_n(0)=M_n(\mathbb{C})/\!/_0 B$. 
Thus the family $f \colon (M_n(\mathbb{C}) \times \mathbb{C})/\!/B \to \mathbb{C}$ can be viewed as a toric degeneration of a flag variety. 
The existence of a toric degeneration of a flag variety is originally proved in \cite{GL,C} in terms of representation theory.

\subsection{Degeneration in stages}\label{stage}

To relate the $U(n)$-action on $Fl_n=Fl_n(1)$ with the $T_{GC}$-action on $Fl_n(0)$, Kogan and Miller introduced degeneration in stages as follows. 
For $\tau=(t_2, \dots, t_n) \in (\mathbb{C}^\times)^{n-1}$ we define $\tau^\omega \in M_n(\mathbb{C}^\times)$ by
$$
(\tau^\omega)_{ij}=t_i^{\omega_{ij}},  ~~\text{ where $t_1=1$ and $\omega_{ij}$ is given in (\ref{omega}).}
$$
In the above $(\tau^\omega)_{ij}$ is the $(i,j)$-component of $\tau^\omega \in M_n(\mathbb{C}^\times)$.
Then we define the deformed action $\bullet_\tau$ of $B$ on $M_n(\mathbb{C})$ by 
$$
V \bullet_\tau b = V \bullet \{ \iota(\tau^\omega) (b,\dots,b) (\iota(\tau^\omega))^{-1} \}. 
$$
Thus we have $Fl_n(\tau)=M_n(\mathbb{C})/\!/_\tau B$ for $\tau \in (\mathbb{C}^\times)^{n-1}$ in the same way as in Subsection \ref{tg}.
We note that $Fl_n(\tau)$ is also embedded in $\mathbb{P}$ by the deformed Pl\"ucker relations as $Fl_n(t)$. 
Set $$\tau_k^t =(\underbrace{1, \dots,1}_{n-1-k},t, \underbrace{0, \dots, 0}_{k-1}) \in \mathbb{C}^{n-1}~~\text{for $t \in [0,1]$ and $k=1, \dots, n-1$.}$$ 
It is easy to see that $Fl_n(\tau_k^t)=M_n(\mathbb{C})/\!/_{\tau_k^t} B$ is well-defined. 
Note that $Fl_n(\tau_k^t)$ has singularities if $\tau_k^t=\tau_1^0$ or $k \ge 2$. 
We call the family $\{ Fl_n(\tau_k^t) \}_{t \in [0,1]}$ the $k$-th stage of the degeneration.
Note that 
\begin{align*}
& U(n-k+1) \times T_{GC}^{(n-1)} \times \dots \times T_{GC}^{(n-k+1)} ~\text{acts on $Fl_n(\tau_k^1)$,} \\
& U(n-k) \times T_{GC}^{(n-1)} \times \dots \times T_{GC}^{(n-k+1)} ~\text{acts on $Fl_n(\tau_k^t)$ for $t \in (0,1)$,} \\
& U(n-k) \times T_{GC}^{(n-1)} \times \dots \times T_{GC}^{(n-k)} ~\text{acts on $Fl_n(\tau_k^0)$}. 
\end{align*}
Kogan and Miller considered the following degeneration in stages:
\begin{align*}\label{stage}
Fl_n=Fl_n & (\tau_1^1) \overset{1st}{\longrightarrow} Fl_n(\tau_1^0)=Fl_n(\tau_2^1) \longrightarrow \dots \\ 
& \longrightarrow Fl_n(\tau_k^1) \overset{k-th}{\longrightarrow} Fl_n(\tau_k^0) \longrightarrow \dots \longrightarrow Fl_n(\tau_{n-1}^0)=Fl_n(0).
\end{align*}

In \cite{NNU} Nishinou, Nohara and Ueda clarified the relation between the Gelfand-Cetlin system on the flag variety $Fl_n$ and the completely integrable system on $Fl_n(0)$ coming from its toric structure as follows. 
The smooth part $Fl_n(\tau_k^t)^{reg}$ of $Fl_n(\tau_k^t)$ has a symplectic structure $\iota_{\tau_k^t}^* \omega_\mathbb{P}$, where $\iota_{\tau_k^t} \colon Fl_n(\tau_k^t)^{reg} \to \mathbb{P}$ is the deformed Pl\"ucker embedding. 
Let $\mu_{U(n-k)} \colon Fl_n(\tau_k^t)^{reg} \to \mathfrak{u}(n-k)$ be the moment map for $U(n-k)$-action on $Fl_n(\tau_k^t)^{reg}$ for $t \in [0,1]$, where $\mathfrak{u}(n-k)$ is identified with $\mathfrak{u}(n-k)^*$ by the invariant inner product. 
Define a map $\lambda_{n-k}^j \colon \mathfrak{u}(n-k) \to \mathbb{R}$ such that $\lambda_{n-k}^1(A) \ge \dots \ge \lambda_{n-k}^{n-k}(A)$ are eigenvalues of $-\sqrt{-1}A$ for $A \in \mathfrak{u}(n-k)$ as in Section \ref{main}. 
Then, in \cite{NNU}, the authors proved the following.
\begin{proposition}\label{nnu}
There exist an open dense subset $Fl_n(\tau_k^t)^\circ \subset Fl_n(\tau_k^t)^{reg}$ and a symplectic diffeomorphism $\varphi_k^{t_2, t_1} \colon Fl_n(\tau_k^{t_1})^\circ \to Fl_n(\tau_k^{t_2})^\circ$ for each $k=1, \dots, n-1$, $t \in [0,1]$ and $0 \le t_2 \le t_1 \le 1$ which satisfy the following:
\\
$(1)$ $Fl_n(\tau_1^1)^\circ= \mu_{GC}^{-1}(\mathrm{Int}\Delta_{GC}) \subset Fl_n$ holds. 
\\
$(2)$ $\varphi_k^{t, t}$ is the identity map for any $t \in [0,1]$. 
Moreover, $\varphi_k^{t_3, t_2} \circ \varphi_k^{t_2, t_1}=\varphi_k^{t_3, t_1}$ holds for $0 \le t_3 \le t_2 \le t_1 \le 1$.
\\
$(3)$ Under the identification of $Fl_n(\tau_k^t)^\circ$ for all $t \in [0,1]$ by the map $\varphi_k^{t_2, t_1}$, the moment map for $U(n-k) \times T_{GC}^{(n-1)} \times \dots \times T_{GC}^{(n-k+1)}$-action on $Fl_n(\tau_k^t)^\circ$ is independent of $t \in (0,1]$.
\newline
$(4)$ $( \lambda_{n-k}^j \circ \mu_{U(n-k)} ~|~ 1 \le j \le n-k) \colon Fl_n(\tau_k^0)^\circ \to \mathbb{R}^{n-k}$ coincides with the moment map for the $T_{GC}^{(n-k)}$-action on $Fl_n(\tau_k^0)^\circ$.
\end{proposition}
The diffeomorphism $\varphi_k^{t_2, t_1} \colon Fl_n(\tau_k^{t_1})^\circ \to Fl_n(\tau_k^{t_2})^\circ$ is constructed by using the gradient-Hamiltonian flow due to Ruan \cite{R}, which is explained in the next section. 
The moment map for $U(n-k) \times T_{GC}^{(n-1)} \times \dots \times T_{GC}^{(n-k+1)}$-action on $Fl_n(\tau_k^t)^\circ$ induces the completely integrable system on $Fl_n(\tau_k^t)^\circ$ in the same way as in the case of the Gelfand-Cetlin system. 
Proposition \ref{nnu} implies the completely integrable system on $Fl_n(\tau_k^t)^\circ$ for $t \in [0,1]$ and $1 \le k \le n-1$ remains the same during the degeneration in stages.  

Due to Proposition \ref{nnu}, we have a diffeomorphism
\begin{equation}\label{psi-0}
\Psi_0 = \varphi_{n-1}^{0,1} \circ \varphi_{n-2}^{0,1} \circ \dots \circ \varphi_{1}^{0,1} \colon Fl_n^\circ \to Fl_n(0)^\circ.
\end{equation}
where $Fl_n^\circ=Fl_n(\tau_1^1)^\circ$ and $Fl_n(0)^\circ=Fl_n(\tau_{n-1}^0)^\circ$. 
Then Nishinou, Nohara and Ueda proved the following. 
\begin{corollary}\label{nnu2}
Let $\mu_{GC} \colon Fl_n \to \mathbb{R}^{\frac{n(n-1)}{2}}$ be the Gelfand-Cetlin system. 
Let $\mu_{T_{GC}} \colon Fl_n(0) \to (\mathfrak{t}_{GC})^*$ be the moment map for the action of $T_{GC}$ on $Fl_n(0)$. 
Then there is a linear isomorphism $i \colon \mathbb{R}^{\frac{n(n-1)}{2}} \to (\mathfrak{t}_{GC})^*$ such that $i \circ \mu_{GC}= \mu_{T_{GC}} \circ \Psi_0 \colon Fl_n^\circ \to  (\mathfrak{t}_{GC})^*$
In particular, $Fl_n(0)^\circ= \mu_{T_{GC}}^{-1}(\mathrm{Int}\Delta_{GC}) \subset Fl_n(0)$ holds. 
\end{corollary}
Therefore the authors concluded that $Fl_n(0)$ is a toric variety constructed from the Gelfand-Cetlin polytope $\Delta_{GC}$. 
This fact is originally proved in \cite{KM} in a different way. 
So $Fl_n(0)$ is called a Gelfand-Cetlin toric variety. 
Moreover, the Gelfand-Cetlin polytope $\Delta_{GC}$ can be considered naturally as a subset of $(\mathfrak{t}_{GC})^*$. 
From now on we consider the Gelfand-Cetlin system to be the map $\mu_{GC} \colon Fl_n \to (\mathfrak{t}_{GC})^*$.
\section{Gradient-Hamiltonian flow}\label{secghf}

Let $(M, \omega, J)$ be a K\"ahler manifold. 
Let $f \colon M \to \mathbb{C}$ be a holomorphic function.
Set $B=f(M)$ and $V_c= f^{-1}(c)$ for $c \in B$.
Denote the inclusion map of $V_c$ by $\rho_c \colon V_c \to M$. 
Then we have a family of symplectic manifolds $\{ (V_c, \rho_c^* \omega) \}_{c \in B_{reg}}$ where $B_{reg}$ is the set of regular values of $f$.
To identify these symplectic manifolds, Ruan introduced the gradient-Hamiltonian flow in \cite{R}.
In this section we recall the gradient-Hamiltonian flow and its basic properties. 
We also discuss the lift of the gradient-Hamiltonian flow to the prequantum line bundle.

By simple computations we see that the following.
\begin{lemma}\label{gr-h-reg-pts} 
Let $(M, \omega, J)$ be a K\"ahler manifold. 
Let $\Re f$ and $\Im f$ be the real and imaginary part of the holomorphic function $f \colon M \to \mathbb{C}$ respectively. 
Let $X_{\Im f} \in \mathcal{X}(M)$ be  the Hamiltonian vector field of the function $\Im f$.
Then the following holds: 
$$X_{\Im f}= -\mathrm{grad}(\Re f), ~~\text{that is,}~~ i(-\mathrm{grad}(\Re f))\omega=-d(\Im f).$$
In particular, $X_{\Im f}= -\mathrm{grad}(\Re f)$ is non-zero at a regular point of $f$.
\end{lemma}

Suppose that $f$ is proper and that each point in $M$ is a regular point of $f$.
Then we have the following vector field 
$$
Z = -\frac{\mathrm{grad}(\Re f)}{|\mathrm{grad}(\Re f)|^2}= \frac{X_{\Im f}}{|X_{\Im f}|^2} \in \mathcal{X}(M).
$$
It is easy to see that 
$$
Z(\Re f)=-1, ~~~Z(\Im f)=0 ~~\text{on $M$. } 
$$
Since $f \colon M \to B$ is proper, for any $c \in B$ there exists $\epsilon_c >0$ such that the flow $\{\varphi_t \}_t$ generated by the vector field $Z \in \mathcal{X}(M)$ induces a diffeomorphism $\varphi_t  |_{V_c}\colon V_c \to V_{c-t}$ for $t \in (-\epsilon_c, \epsilon_c)$. 
In \cite{R} Ruan found the following remarkable property.
\begin{proposition}\label{ghf}
$(\varphi_t |_{V_c})^* (\rho_{c-t}^*\omega)=\rho_c^* \omega$ for $t \in (-\epsilon_c, \epsilon_c)$. 
\end{proposition}
We call $Z \in \mathcal{X}(M)$ the gradient-Hamiltonian vector field, and $\{\varphi_t \}_t$ the gradient-Hamiltonian flow respectively.

Next we discuss the lift of the gradient-Hamiltonian flow to the prequantum line bundle. 
Let us assume that there exists a prequantum line bundle $(L, h, \nabla)$ on $M$ in addition to the above setting. 
For any $c \in B$ we denote the restriction of $(L, h, \nabla)$ to the fiber $V_c$ by $(L^{V_c}, h^{V_c}, \nabla^{V_c})$.

The horizontal lift $\tilde{Z} \in \mathcal{X}(L)$ of $Z \in \mathcal{X}(M)$ induces the flow  $\{ \tilde{\varphi}_t \}_t$, which is a lift of the gradient-Hamiltonian flow $\{\varphi_t \}_t$. 
Similarly, for any $c \in B$ there exists $\epsilon_c >0$ such that the flow $\{\tilde{\varphi}_t \}_t$ induces a bundle isomorphism $\tilde{\varphi}_t  |_{L^{V_c}}\colon L^{V_c} \to L^{V_{c-t}}$ for $t \in (-\epsilon_c, \epsilon_c)$. 

Then we have the following proposition. 
Since its proof does not seem to be found in the literature, we give a proof here.
\begin{proposition}\label{liftghf}
$(\tilde{\varphi}_t |_{L^{V_c}})^* \nabla^{V_{c-t}}=\nabla^{V_c}$ and $(\tilde{\varphi}_t |_{L^{V_c}})^* h^{V_{c-t}}=h^{V_c}$ for $t \in (-\epsilon_c, \epsilon_c)$. 
\end{proposition}
\begin{proof}
Since the connection $\nabla$ preserves the Hermitian metric $h$, the second assertion is obvious.  
So we prove the first assertion.

Since $Z(\Im f)=0$ on $M$, the gradient-Hamiltonian flow $\{\varphi_t \}_t$ preserves $M_{\Im f = \Im c} = \{ p \in M ~|~ \Im f(p)=\Im c \}$. 
First we show that $i(Z)\omega=0$ on $M_{\Im f = \Im c}$. 
In fact, we have
$$
i(Z)\omega=i(\frac{X_{\Im f}}{|X_{\Im f}|^2})\omega= \frac{-d(\Im f)}{|X_{\Im f}|^2}=0 ~~\text{on $M_{\Im f = \Im c}$.}
$$

Let $S \subset L$ be the unit sphere bundle and $p \colon S \to M$ the projection.
If we denote the connection form of $\nabla$ by $\alpha \in \Omega^1(S)$, then we have $d \alpha =p^* \omega$.
Since the restriction of the horizontal lift $\tilde{Z} \in \mathcal{X}(L)$ to $S$ can be considered as $\tilde{Z} \in \mathcal{X}(S)$, we have $i(\tilde{Z}) \alpha =0$ and $p_* \tilde{Z} = Z$.
So,  on $p^{-1}(M_{\Im f = \Im c})$, we have 
$$
L_{\tilde{Z}} \alpha 
= i(\tilde{Z})(p^* \omega) =  p^* \{ i(p_* \tilde{Z}) \omega \} = p^* \{ i(Z) \omega \}=0 .
$$
Thus the flow induced by the vector field $\tilde{Z} \in \mathcal{X}(S)$ preserves the connection $\nabla$ on $p^{-1}(M_{\Im f = \Im c})$.
\end{proof}
\section{Toric K\"ahler structures of toric manifolds}\label{toric}
In this section we review toric K\"ahler structures of toric manifolds.  
Starting from a Delzant polytope, we construct a symplectic toric manifold in Subsection \ref{symp-toric-mfld} and a complex toric manifold in Subsection \ref{comp-toric-mfld}. 
We identify them according to a choice of symplectic potentials due to \cite{Ab1,Ab2, Gu1,Gu2} in Subsection \ref{symp-potl}. 
We also review certain deformation of toric K\"ahler structures by changing symplectic potentials, which was introduced in \cite{BFMN}. 

Let $T^n$ be a real torus with the Lie algebra $\mathfrak{t}^n$. 
Let 
\begin{equation}\label{polytope}
\Delta= \{ p \in (\mathfrak{t}^n)^* ~|~ \langle p, r_j \rangle + \lambda_j \ge 0 ~~\text{for $j=1, \dots, d$} \}
\end{equation}
be a bounded Delzant polytope, where $\langle ~, ~ \rangle \colon (\mathfrak{t}^n)^* \times \mathfrak{t}^n \to \mathbb{R}$ is the natural pairing and $r_j$ is a primitive vector in the lattice $\mathfrak{t}^n_{\mathbb{Z}}$ for $j=1, \dots, d$.
We assume $\lambda_1, \dots, \lambda_d \in \mathbb{Z}$. 
We set 
\begin{equation}\label{wall}
l_j(p)= \langle p, r_j \rangle + \lambda_j, ~~~~~~ F_j = \{ p \in (\mathfrak{t}^n)^* ~|~ l_j(p) = 0 \} ~~~~~~ \text{for $j=1, \dots, d$.}
\end{equation}

Let $T^d$ be a real torus with the Lie algebra $\mathfrak{t}^d$ and $X_1, \dots, X_d \in \mathfrak{t}^d_{\mathbb{Z}}$ be the standard basis of $\mathfrak{t}^d$. 
Let $\pi \colon \mathfrak{t}^d \to \mathfrak{t}^n$ be the surjective Lie algebra homomorphism defined by $\pi(X_j)=r_j$ for $j=1, \dots, d$.  
Then the kernel of the corresponding Lie group homomorphism $\tilde{\pi} \colon T^d \to T^n$ is a connected subtorus $K$ of $T^d$ with the Lie algebra $\mathfrak{k}$. 
Let $u_1, \dots, u_d \in (\mathfrak{t}^d)^*$ be the dual basis of $X_1, \dots, X_d \in \mathfrak{t}^d_{\mathbb{Z}}$.  
We set $\lambda_\Delta= \lambda_1 u_1 + \dots + \lambda_d u_d \in (\mathfrak{t}^d)^*_{\mathbb{Z}}$. 

\subsection{A symplectic toric manifold $M_{symp}$}\label{symp-toric-mfld}
Let $\tilde{\omega}$ be the standard symplectic form on $\mathbb{C}^d$. 
The natural action of $T^d$ on $(\mathbb{C}^d, \tilde{\omega})$ admits a moment map $\mu_{T^d} \colon \mathbb{C}^d \to (\mathfrak{t}^d)^*$, given by $\mu_{T^d}(z)=\pi\sum_{j=1}^d|z_j|^2u_j$, where $z=(z_1, \dots, z_d)$. 
If we denote the dual map of the inclusion $\iota \colon \mathfrak{k} \to \mathfrak{t}^d$ by $\iota^* \colon (\mathfrak{t}^d)^* \to \mathfrak{k}^*$, then the moment map $\mu_K \colon \mathbb{C}^d \to \mathfrak{k}^*$ for the action of the subtorus $K$ on $(\mathbb{C}^d, \tilde{\omega})$ is given by $\mu_K(z)=\pi\sum_{j=1}^d|z_j|^2 \iota^* u_j$.
The compact symplectic toric manifold $M_{symp}$ is defined to be the symplectic quotient $M_{symp}=\mu_K^{-1}(\iota^* \lambda_\Delta)/K$ with the natural symplectic structure $\omega \in \Omega^2(M_{symp})$. 
The quotient torus $T^n=T^d/K$ acts on $(M_{symp}, \omega)$ with the moment map $\mu_{T^n} \colon M_{symp} \to (\mathfrak{t}^n)^*$.
Since $\mu_{T^d}(z)-\lambda_\Delta \in \ker\{\iota^* \colon (\mathfrak{t}^d)^* \to \mathfrak{k}^* \} = \mathrm{image} \{\pi^* \colon (\mathfrak{t}^n)^* \to (\mathfrak{t}^d)^* \}$, it is given by $\mu_{T^n}([z]) = (\pi^*)^{-1}(\mu_{T^d}(z)-\lambda_\Delta) \in (\mathfrak{t}^n)^*$. 
It is well known that $\mu_{T^n}(M_{symp})= \Delta$.

Next we define a prequantum line bundle on $M_{symp}$. 
Let $\tilde{L}_{symp}=\mathbb{C}^d \times \mathbb{C}$ be the trivial line bundle with the standard fiber metric $\tilde{h}$. 
Let $\tilde{\nabla}$ be a Hermitian connection on $\tilde{L}_{symp}$ defined by $\tilde{\nabla}=d - \sqrt{-1}\pi \sum_{i=j}^d (x_j dy_j -y_j dx_j)$, where $x_j, y_j$ are the real and imaginary part of $z_j$ respectively. 
The action of $T^d$ on $\tilde{L}_{symp}$ defined by $(z,v)\mathrm{Exp}_{T^d} \xi = (z \mathrm{Exp}_{T^d}\xi, ve^{2 \pi \sqrt{-1} \langle \lambda_{\Delta}, \xi \rangle})$ preserves the Hermitian metric $\tilde{h}$ and the connection $\tilde{\nabla}$, where $\mathrm{Exp}_{T^d} \colon \mathfrak{t}^d \to T^d$ is the exponential map. 
Then the prequantum line bundle $(L_{symp}, h, \nabla)$ on $M_{symp}$ is defined to be the quotient of the restriction of $\tilde{L}_{symp}$ to $\mu_K^{-1}(\iota^* \lambda_\Delta)$ by the action of the subtorus $K$. 
Moreover, the quotient torus $T^n=T^d/K$ acts on $L_{symp}$, preserving $h$ and $\nabla$. 
Let $[z]_K \in M_{symp}$ denote a point represented by $z \in \mu_K^{-1}(\iota^* \lambda_\Delta)$. 
Similarly $[z,v]_K$ denotes a point in $L_{symp}$ represented by $(z,v) \in \mu_K^{-1}(\iota^* \lambda_\Delta) \times \mathbb{C}$.

Set $M_{symp}^0 = \mu_{T^n}^{-1}(\mathrm{Int}\Delta)$, where $\mathrm{Int}\Delta$ is the interior of the Delzant polytope $\Delta$. 
Then it is easy to see that 
$(\sqrt{\frac{l_1(p)}{\pi}}, \dots, \sqrt{\frac{l_d(p)}{\pi}}) \in \mu_K^{-1}(\iota^* \lambda_\Delta)$ for any $p \in \mathrm{Int}\Delta$.
Therefore the map $\psi^0_{symp} \colon \mathrm{Int}\Delta \times \mathfrak{t}^n/\mathfrak{t}^n_\mathbb{Z} \to M_{symp}^0$ defined by \begin{align}\label{symp-coord}
\psi^0_{symp} (p, [q]) &=[(\sqrt{\frac{l_1(p)}{\pi}}, \dots, \sqrt{\frac{l_d(p)}{\pi}})]_K \mathrm{Exp}_{T^n}(q) \\
&= [(\sqrt{\frac{l_1(p)}{\pi}} e^{2 \pi \sqrt{-1} \langle u_1, \tilde{q} \rangle}, \dots, \sqrt{\frac{l_d(p)}{\pi}}e^{2 \pi \sqrt{-1} \langle u_d, \tilde{q} \rangle})]_K \nonumber
\end{align}
is a diffeomorphism, 
where $\tilde{q} \in \mathfrak{t}^d$ is taken so that  $\pi(\tilde{q})=q$. 
Note that we have 
\begin{equation}\label{mm}
\mu_{T^n} \circ \psi^0_{symp}(p, [q])= p \hspace*{5mm}\text{for $(p, [q]) \in \mathrm{Int}\Delta \times \mathfrak{t}^n/\mathfrak{t}^n_\mathbb{Z}$.}
\end{equation}

Next we define a section $s^0_{symp}$ of $L_{symp}$ restricted to $M_{symp}^0$ by 
$$
s_{symp}^0(p,[q])=[(\sqrt{\frac{l_1(p)}{\pi}}, \dots, \sqrt{\frac{l_d(p)}{\pi}}),1]_K \mathrm{Exp}_{T^n}(q) \in L_{symp}.
$$ 
This section induces a unitary trivialization of the prequantum line bundle $L_{symp}$ on $M_{symp}^0$. 

Fix a $\mathbb{Z}$-basis $p_1, \dots, p_n \in (\mathfrak{t}^n)^*_\mathbb{Z}$ and its dual basis $q_1, \dots, q_n \in \mathfrak{t}^n_\mathbb{Z}$. 
Set $\Delta^0= \{ x=(x_1, \dots, x_n) \in \mathbb{R}^n ~|~ \sum_{i=1}^n x_i p_i \in \mathrm{Int}\Delta \}$. 
Then we have a coordinate $(x,[\theta])\in \Delta^0 \times \mathbb{R}^n/\mathbb{Z}^n$ on $\mathrm{Int}\Delta \times \mathfrak{t}^n/\mathfrak{t}^n_\mathbb{Z}$. 
So $(x,[\theta])\in \Delta^0 \times \mathbb{R}^n/\mathbb{Z}^n$ can be considered as a coordinate on $M_{symp}^0$.
It is easy to see the following by simple computations.
\begin{lemma}\label{cansymp}
Let $(x, [\theta]) \in \Delta^0 \times \mathbb{R}^n/\mathbb{Z}^n$ be the coordinate on $M_{symp}^0$ induced by the fixed basis $p_1, \dots, p_n \in (\mathfrak{t}^n)^*_\mathbb{Z}$.
Then the symplectic form $\omega$ on $M_{symp}^0$ and the connection $\nabla$ on $L_{symp}|_{M_{symp}^0}$ are described as follows. 
\\
$(1)$ $\omega|_{M_{symp}^0}= \sum_{i=1}^n dx_i \wedge d \theta_i$.
\\
$(2)$ $\nabla|_{M_{symp}^0}=d- 2 \pi \sqrt{-1} \sum_{i=1}^n x_i d \theta_i$ with respect to the unitary trivialization defined by the section $s_{symp}^0$ on $M_{symp}^0$.
\\
$(3)$ For $m \in \mathrm{Int}\Delta$, $\mu_{T^n}^{-1}(m)$ is a Bohr-Sommerfeld fiber for the prequantum line bundle $(L_{symp}, h, \nabla)$ if and only if $m \in \mathrm{Int}\Delta \cap (\mathfrak{t}^n)^*_{\mathbb{Z}}$. 
Moreover, $\delta_m([\theta])=e^{2 \pi \sqrt{-1} \sum_{i=1}^n m_i \theta_i} s_{symp}^0 |_{\mu_{T^n}^{-1}(m)}$ is a covariantly constant section of $(L_{symp}, h, \nabla)|_{\mu_{T^n}^{-1}(m)}$ for $m= \sum_{i=1}^n m_i p_i \in \mathrm{Int}\Delta \cap (\mathfrak{t}^n)^*_{\mathbb{Z}}$, where $[ \theta ] \in \mathbb{R}^n/\mathbb{Z}^n$ is a coordinate on $\mu_{T^n}^{-1}(m)$.
\end{lemma}
\subsection{A complex toric manifold $M_{comp}$}\label{comp-toric-mfld}
Let $\Delta$ be a Delzant polytope defined by (\ref{polytope}), and denote its set of vertices by $\Delta(0)$. 
Let $F_j \subset (\mathfrak{t}^n)^*$ be the hyperplane defined in (\ref{wall}) for $j=1, \dots, d$.  
For each $v \in \Delta(0)$ we set $\Lambda_v= \{ j ~|~ v \in F_j \}$, $\mathbb{C}^d_v = \{ z \in \mathbb{C}^d ~|~ z_j \ne 0 ~~\text{if $j \in \{1, \dots, d \} \setminus \Lambda_v$} \}$ and $\mathbb{C}^d_\Delta =\bigcup_{v \in \Delta(0)}\mathbb{C}^d_v$.
Then the compact complex toric manifold $M_{comp}$ is defined to be the quotient space $M_{comp}=\mathbb{C}^d_\Delta/K_\mathbb{C}$, where $K_\mathbb{C}$ is the complexification of the subtorus $K$. 
Similarly the complexification of the torus $T^d$ is denoted by $T^d_\mathbb{C}$.
The quotient torus $T^n_\mathbb{C}=T^d_\mathbb{C}/K_\mathbb{C}$ acts on $M_{comp}$, preserving its complex structure $J$.

Next we define a holomorphic line bundle on $M_{comp}$.
Let $\tilde{L}_{comp}= \mathbb{C}^d \times \mathbb{C}$ be a trivial holomorphic line bundle on $\mathbb{C}^d$. 
Define the action of $T^d_\mathbb{C}$ on $\tilde{L}_{comp}$ by $(z,v)\mathrm{Exp}_{T^d_{\mathbb{C}}}\xi = (z \mathrm{Exp}_{T^d_{\mathbb{C}}} \xi, v e^{2 \pi \sqrt{-1} \langle \lambda_\Delta, \xi \rangle} )$. 
The holomorphic line bundle $L_{comp}$ is defined to be the quotient of the restriction of $\tilde{L}_{comp}$ to $\mathbb{C}^d_\Delta $ by the action of $K_\mathbb{C}$. 
Then the quotient torus $T^n_\mathbb{C}=T^d_\mathbb{C}/K_\mathbb{C}$ acts on $L_{comp}$, preserving its holomorphic structure $\bar{\partial}$. 
Let $[z]_{K_\mathbb{C}} \in M_{comp}$ denote a point represented by $z \in \mathbb{C}^d_\Delta$. 
Similarly $[z,v]_{K_\mathbb{C}}$ denotes a point in $L_{comp}$ represented by $(z,v) \in \mathbb{C}^d_\Delta \times \mathbb{C}$.

Next we define a meromorphic section $s^0_{comp}$ of $L_{comp}$ on $M_{comp}$ by 
$$ 
s_{comp}^0([z]_{K_\mathbb{C}}) =[z, \prod_{j=1}^d z_j^{\lambda_j}]_{K_\mathbb{C}} \in L_{comp}~~ \text{for $ z \in \mathbb{C}^d_\Delta $}.
$$ 
The section $s^0_{comp}$ is holomorphic and non-zero on $M_{comp}^0 = (\mathbb{C}^\times)^d / K_\mathbb{C}$, where $(\mathbb{C}^\times)^d=\{ z \in \mathbb{C}^d ~|~ z_i \ne 0 ~~ \text{for $i=1, \dots, d$} \} \subset \mathbb{C}^d_\Delta$.
So it induces a holomorphic trivialization of $L_{comp}$ on $M_{comp}^0$. 

For $m \in \Delta \cap (\mathfrak{t}^n)^*_{\mathbb{Z}}$ we define a holomorphic section $\sigma^m$ of $L_{comp}$ by 
\begin{align}\label{def-sigma-m-0}
\sigma^m ([z]_{K_\mathbb{C}})=[z, \prod_{j=1}^d z_j^{l_j(m)}]_{K_\mathbb{C}} \in L_{comp}~~ \text{for $ z \in \mathbb{C}^d_\Delta $.}
\end{align}
It is well known that $\{ \sigma^m ~|~ m \in \Delta \cap (\mathfrak{t}^n)^*_{\mathbb{Z}} \}$ is a basis of the space of holomorphic sections $H^0(L_{comp}, \bar{\partial})$.

Next we introduce a complex coordinate on $M_{comp}^0$. 
Fix a $\mathbb{Z}$-basis $p_1, \dots, p_n \in (\mathfrak{t}^n)^*_\mathbb{Z}$ and its dual basis $q_1, \dots, q_n \in \mathfrak{t}^n_\mathbb{Z}$ as in Subsection \ref{symp-toric-mfld}. 
Then we define a complex coordinate $\varphi^0_{comp} \colon M_{comp}^0 \to (\mathbb{C}^\times)^n$ by 
\begin{align}\label{comp-coord}
\varphi ^0_{comp}([z]_{K_\mathbb{C}})= (\prod_{j=1}^d z_j^{\langle p_1, r_j \rangle}, \dots, \prod_{j=1}^d z_j^{\langle p_n, r_j \rangle}),
\end{align}
where $r_j \in \mathfrak{t}^n_\mathbb{Z}$ is the vector in (\ref{polytope}) for $j=1, \dots, d$. 
Since $\prod_{j=1}^d z_j^{\langle p_i, r_j \rangle}$ is a $K_\mathbb{C}$-invariant meromorphic function on $\mathbb{C}^d$, it descends to a meromorphic function on $M_{comp}$.
If we set $(w_1, \dots,w_n)=\varphi ^0_{comp}([z]_{K_\mathbb{C}})$, then we have
\begin{align}\label{def-sigma-m}
\sigma^m & ([z]_{K_\mathbb{C}}) 
=(\prod_{i=1}^n w_i^{\langle m, q_i \rangle })s_{comp}^0([z]_{K_\mathbb{C}})~~\text{on $M_{comp}$.} 
\end{align}

\subsection{Symplectic potentials}\label{symp-potl}

In Subsections \ref{symp-toric-mfld} and \ref{comp-toric-mfld}, starting from a Delzant polytope $\Delta$ defined in (\ref{polytope}), we constructed a symplectic and complex toric manifold respectively.  
In this section we identify them, using symplectic potentials due to \cite{Gu1,Gu2,Ab1,Ab2}. 
We also recall a certain deformation of toric K\"ahler structures due to \cite{BFMN}.

The inclusion $\mu_K^{-1}(\iota^* \lambda_\Delta) \subset \mathbb{C}^d_\Delta$ induces a map $\chi_{can} \colon M_{symp} \to M_{comp}$.  
It is well known that this map is a diffeomorphism.
In \cite{Gu1,Gu2} Guillemin showed that this map is described by a single function $g_{can}$ as follows. 

Fix a $\mathbb{Z}$-basis $p_1, \dots, p_n \in (\mathfrak{t}^n)^*_\mathbb{Z}$ and its dual basis $q_1, \dots, q_n \in \mathfrak{t}^n_\mathbb{Z}$ as in Subsections \ref{symp-toric-mfld} and \ref{comp-toric-mfld}. 
Fix $\tilde{q}_i \in \mathfrak{t}^d_\mathbb{Z}$ so that $\pi(\tilde{q}_i)=q_i$ for $i=1, \dots, n$. 
Let $(x,[\theta])$ be the symplectic coordinate on $M_{symp}^0$ and $(w_1, \dots, w_n)$ the complex coordinate on $M_{comp}^0$ induced by $p_1, \dots, p_n \in (\mathfrak{t}^n)^*_\mathbb{Z}$ respectively. 
If we write $p= \sum_{i=1}^n x_i p_i$, then, by (\ref{symp-coord}) and (\ref{comp-coord}) we have 
$$
w_i (\chi_{can}(x, [\theta]))
= \prod_{j=1}^d (\sqrt{\frac{l_j(p)}{\pi}} e^{2\pi \sqrt{-1} \sum_{l=1}^n \langle u_j, \tilde{q}_l \rangle \theta_l})^{\langle p_i, r_j \rangle}
= e^{2\pi (\frac{\partial g_{can}}{\partial x_i} + \sqrt{-1} \theta_i)},
$$
where $ g_{can} \colon \mathrm{Int}\Delta \to \mathbb{R}$ is a function defined by 
$$ g_{can}(p)=\frac{1}{4 \pi } \sum_{j=1}^d l_j(p) \log l_j(p)+ ( \text{ a linear function on $(\mathfrak{t}^n)^*$ } ) ~~ \text{ for $ p \in \mathrm{Int} \Delta $ }.
$$ 
Note that $ g_{can} $ extends continuously to $ g_{can} \colon \Delta \to \mathbb{R} $. 

\begin{definition}
{\rm A function $g \in C^0 (\Delta)$ is a {\it symplectic potential} if and only if the following holds:
\newline
$(1)$ $g - g_{can} \in C^\infty(\Delta)$,
\newline
$(2)$ The Hessian $\mathrm{Hess}_p g$ of $g$ at $p$ is positive definite for any $p \in \mathrm{Int}\Delta$,
\newline
$(3)$ there exists a strictly positive function $\alpha \in C^\infty (\Delta)$ such that 
$$
\det(\mathrm{Hess}_p g)=[\alpha(p) \prod_{j=1}^d l_j(p)]^{-1} \hspace*{5mm} \text{for any $ p \in \mathrm{Int}\Delta $ .}
$$
\newline
The set of symplectic potentials is denoted by $SP(\Delta)$.}
\end{definition}

The following results are due to \cite{Gu1,Gu2, Ab1,Ab2}, supplemented by \cite{BFMN}.

\begin{theorem}\label{GuAb}
Let $\Delta \subset (\mathfrak{t}^n)^* $ be a Delzant polytope.
Let $(M_{symp},\omega)$  be a symplectic toric manifold and $(M_{comp},J)$ a complex toric manifold constructed from $\Delta$. 
Let $(L_{symp},h, \nabla)$ be a prequantum line bundle on $M_{symp}$ and $(L_{comp},\bar{\partial})$ a holomorphic line bundle on $M_{comp}$ constructed from $\Delta$.
Fix a $\mathbb{Z}$-basis $p_1, \dots, p_n \in (\mathfrak{t}^n)^*_\mathbb{Z}$. 
Let $(x,[\theta])$ be the symplectic coordinate on $M_{symp}^0$ and $w=(w_1, \dots, w_n)$ the complex coordinate on $M_{comp}^0$ induced by $p_1, \dots, p_n \in (\mathfrak{t}^n)^*_\mathbb{Z}$ respectively. 
\newline
$(A)$ Each $g \in SP(\Delta)$ defines a $T^n$-equivariant diffeomorphism $\chi_g \colon M_{symp} \to M_{comp}$ and a $T^n$-equivariant bundle isomorphism $\tilde{\chi}_g \colon L_{symp} \to L_{comp}$ such that the following holds:  
\newline
$(a1)$ The following diagram commutes: 
$$
\begin{array}{ccc}
(L_{symp},h, \nabla) & \overset{\tilde{\chi}_g}{\longrightarrow } & (L_{comp},\bar{\partial}) \\
\downarrow &                                           & \downarrow \\
(M_{symp},\omega) & \overset{\chi_g}{\longrightarrow }         & (M_{comp},J)
\end{array}
$$ 
$(a2)$ $(M_{symp},\omega, \chi_g^*J)$ is a K\"ahler manifold.
\newline
$(a3)$ $\nabla$ is the Chern connection of the Hermitian holomorphic line bundle $(L_{symp}, h, \tilde{\chi}_g^* \bar{\partial})$.
\newline
$(a4)$ $\chi_g|_{M_{symp}^0} \colon M_{symp}^0 \to M_{comp}^0$ is a diffeomorphism given by 
\begin{equation}\label{chi-g}
w_i(\chi_g(x,[\theta]))=e^{2 \pi (\frac{\partial g}{\partial x_i}+ \sqrt{-1} \theta_i)}~~\text{for $i=1, \dots, n$.}
\end{equation}
The map $\chi_g$ is independent of the choice of the basis $p_1, \dots, p_n \in (\mathfrak{t}^n)^*_\mathbb{Z}$.
Moreover, if we write $w_i=e^{2 \pi (y_i+ \sqrt{-1}\theta_i)}$ for $i=1, \dots, n$, then the inverse mapping $(\chi_g|_{M_{symp}^0})^{-1} \colon M_{comp}^0 \to M_{symp}^0$ is given by 
\begin{equation}\label{chi-g-inv}
x_i(\chi_g^{-1}(w))=\frac{\partial f}{\partial y_i}, ~~\theta_i((\chi_g^{-1})(w))=\theta_i ~~\text{for $i=1, \dots, n$,} 
\end{equation} 
where $f(y)=-g(x(y)) + \sum_{i=1}^n x_i(y)y_i$.
\newline
$(a5)$ 
$\displaystyle \tilde{\chi}_g^* s_{comp}^0 = e^{2\pi (g-\sum_{i=1}^n x_i \frac{\partial g}{\partial x_i} )}s_{symp}^0$ on $M_{symp}^0$.
\newline
$(B)$ On the other hand, if $\chi \colon M_{symp} \to M_{comp}$ is a $T^n$-equivariant diffeomorphism such that $(M_{symp},\omega, \chi^*J)$ is a K\"ahler manifold and that $\chi$ is homotopic to $\chi_{can}$, then there exists $g \in SP(\Delta)$ such that $\chi=\chi_g$. 
\end{theorem}

In \cite{BFMN} the authors considered a certain 1-parameter family of symplectic potentials, which provides a 1-parameter family of identification of a symplectic toric manifold with a complex toric manifold.
In other words, it provides a deformation of toric K\"ahler structures. 
The authors proved the following remarkable property of the deformation. 

\begin{proposition}\label{thm:bfmn}
Let $\chi_s \colon M_{symp} \to M_{comp}$ and $\tilde{\chi_s} \colon L_{symp} \to L_{comp}$ be the diffeomorphism and the bundle isomorphism defined by $g_s =g_0 + s \nu \in SP(\Delta)$ for $s \ge 0$ respectively, where $\nu \colon \Delta \to \mathbb{R}$ is a smooth strictly convex function.   
Then, for each $m  \in \Delta \cap (\mathfrak{t}^n)^*_{\mathbb{Z}}$, the section $\frac{\tilde{\chi}_s^*\sigma^m}{\Vert \tilde{\chi}_s^*\sigma^m \Vert_{L^1(M_{symp})}}$ converges to a delta-function section supported on the fiber $\mu_{T^n}^{-1}(m)$ in the following sense: 
there exists a covariantly constant section $\delta_m$ of $(L_{symp},h,\nabla)|_{\mu_{T^n}^{-1}(m)}$ and a measure $d \theta_m$ on $\mu_{T^n}^{-1}(m)$ such that, for any smooth section $\phi$ of the dual line bundle $L_{symp}^*$, the following holds
$$
\lim_{s \to \infty} \int_{M_{symp}} \biggl\langle \phi, \frac{\tilde{\chi}_s^*\sigma^m}{\Vert \tilde{\chi}_s^*\sigma^m \Vert_{L^1(M_{symp})}} \biggr\rangle \frac{\omega^n}{n !} = \int_{\mu_{T^n}^{-1}(m)} \langle \phi , \delta_m \rangle d \theta_m.
$$ 
\end{proposition}
Note that the authors proved the above results not only for $m  \in \mathrm{Int}\Delta \cap (\mathfrak{t}^n)^*_{\mathbb{Z}}$ but for all $m  \in \Delta \cap (\mathfrak{t}^n)^*_{\mathbb{Z}}$.
In Proposition \ref{sigma-m} below we slightly generalize this proposition.
\section{Submanifolds under the deformation due to \cite{BFMN}}\label{sec-identify} 
In the last section, starting from a Delzant polytope $\Delta$ defined by (\ref{polytope}), we constructed a symplectic toric manifold $(M_{symp},\omega)$ and a complex toric manifold $(M_{comp},J)$. 
In this section, we study the change of the identification $\chi_s \colon (M_{symp},\omega) \to (M_{comp},J)$  and its lift $\tilde{\chi}_s \colon (L_{symp}, h,\nabla) \to (L_{comp}, \overline{\partial})$ induced by a family of symplectic potentials $g_s=s_0 + s \nu \in SP(\Delta)$ for $s \ge 0$, where $\nu \in C^\infty(\Delta)$ is a weakly convex function. 
In particular, we study the behavior of submanifolds and the prequantum line bundle on it under the change of identification of the ambient toric manifolds. 

\subsection{Identification of submanifolds} 
Given a complex submanifold $V_{comp}$ of $(M_{comp},J)$, we consider the change of the identification $\chi_s \colon (M_{symp},\omega) \to (M_{comp},J)$. 
This implies that the complex structure of the complex submanifold remains the same, but the symplectic structure $(\chi_s^{-1})^* \omega$ on it changes. 
In this subsection, we develop a method to identify $(V_{comp}, (\chi_s^{-1})^* \omega)$ as a symplectic manifold. 
We also lift the identification to the prequantum line bundle.
\begin{proposition}\label{identify1} 
Let $(V_{comp},J^V)$ be a compact complex submanifold of $(M_{comp},J)$ and $\rho_{comp} \colon V_{comp} \to M_{comp}$ the embedding. 
Set $V_{symp}= \chi_0^{-1}(V_{comp})$, and denote the embedding by $\rho_0 \colon V_{symp} \to M_{symp}$. 
\\
$(1)$ There exists an embedding $\rho_s \colon V_{symp} \to M_{symp}$ such that $\rho_s^* \omega= \rho_0^* \omega$ for each $s \ge 0$.
\\ 
$(2)$ There exists a diffeomorphism $\underline{\chi}_s \colon V_{symp} \to V_{comp}$ such that, for each $s \ge 0$, $(V_{symp}, \rho_0^* \omega, \underline{\chi}_s^* J^V)$ is a K\"ahler manifold and the following diagram commutes:  
$$
\begin{array}{ccc}
(M_{symp},\omega) & \overset{\chi_s}{\longrightarrow } & (M_{comp}, J)\\
\uparrow \rho_s &                                           & \uparrow \rho_{comp}  \\
(V_{symp}, \rho_0^* \omega)  & \overset{\underline{\chi}_s}{\longrightarrow }         & (V_{comp}, J^V).
\end{array}
$$
$(3)$ The maps $\rho_s$ and $\underline{\chi}_s$ are canonically defined and depend smoothly on $s \ge 0$.
\end{proposition}
\begin{proof}
$(1)$ If we set $\psi_s= \chi_0 \circ (\chi_s)^{-1} \colon M_{comp} \to M_{comp}$ and $\omega_s= ((\chi_s)^{-1})^* \omega \in \Omega^2(V_{comp})$ for each $s \ge 0$, then we have 
\begin{equation}\label{omega-s}
\psi_s^* \omega_0=\omega_s.
\end{equation}
We show the following. 
\begin{claim}\label{claim1}
There exists a diffeomorphism $\phi_s \colon V_{comp} \to V_{comp}$ for each $s \ge 0$, such that $\phi_0 =id_{V_{comp}}$ and $\phi_s^* (\rho_{comp}^* \omega_s)= \rho_{comp}^* \omega_0$.
\end{claim}
\begin{proof}[Proof of Claim \ref{claim1}] 
Define a vector field $X_s \in \mathcal{X}(M_{comp})$ by 
\begin{equation}\label{X}
(X_s)_{\psi_s(p)}=\frac{d}{dt}\Big\vert_{t=0}\psi_{s+t}(p) \in T_{\psi_s(p)}M_{comp} ~~\text{for $p \in M_{comp}$}.
\end{equation}
By (\ref{omega-s}) we have
\begin{equation}\label{eta}
\frac{d \omega_s}{ds}=\psi_s^*(L_{X_s} \omega_0)=d \eta_s, ~~\text{where $\eta_s = \psi_s^* \{ i(X_s)\omega_0 \} \in \Omega^1(M_{comp})$.}
\end{equation}

Suppose that there exists a diffeomorphism $\phi_s \colon V_{comp} \to V_{comp}$ such that $\phi_s^* (\rho_{comp}^* \omega_s)= \rho_{comp}^* \omega_0$ for each $s \ge 0$. 
If we define a vector field $Y_s \in \mathcal{X}(V_{comp})$ by 
$
(Y_s)_{\phi_s(p)}=\frac{d}{dt}\big\vert_{t=0}\phi_{s+t}(p) \in T_{\phi_s(p)}V_{comp} ~~\text{for $p \in V_{comp}$},
$
then we have
\begin{align*}
0=\frac{d}{ds}\{ \phi_s^*(\rho_{comp}^* \omega_s) \} 
&= \phi_s^* \{ L_{Y_s}(\rho_{comp}^* \omega_s)+ \frac{d \rho_{comp}^* \omega_s}{ds} \} \\
&= \phi_s^* d \{ i(Y_s)(\rho_{comp}^* \omega_s) + \rho_{comp}^* \eta_s \}.
\end{align*}
Therefore, if we define $Y_s \in \mathcal{X}(V_{comp})$ conversely by 
\begin{equation}\label{Y}
i(Y_s)(\rho_{comp}^* \omega_s) + \rho_{comp}^* \eta_s =0,
\end{equation}
then we have a desired diffeomorphism $\phi_s \colon V_{comp} \to V_{comp}$ by integrating $Y_s \in \mathcal{X}(V_{comp})$. 
Moreover we have $\phi_0 =id_{V_{comp}}$ from this construction.  
\end{proof}

Since $\rho_{comp} \circ \chi_0 |_{V_{symp}}= \chi_0 \circ \rho_0$, we have $(\chi_0 |_{V_{symp}})^* (\rho_{comp}^* \omega_0)= \rho_0^* \omega$. 
Define a smooth map $\rho_s \colon V_{symp} \to M_{symp}$ by $\rho_s =(\chi_s)^{-1} \circ \rho_{comp} \circ \phi_s \circ \chi_0 |_{V_{symp}}$.
By Claim \ref{claim1} we have 
$$
\rho_s^* \omega= (\chi_0 |_{V_{symp}})^* \phi_s^* (\rho_{comp}^*\omega_s)=(\chi_0 |_{V_{symp}})^* (\rho_{comp}^* \omega_0)=\rho_0^* \omega.
$$ 
$(2)$ Define $\underline{\chi}_s \colon V_{symp} \to V_{comp}$ by $\underline{\chi}_s = \phi_s \circ \chi_0 |_{V_{symp}}$.
Then we have $\chi_s \circ \rho_s = \rho_{comp} \circ \underline{\chi}_s$. 
Since $\underline{\chi}_s^* (\rho_{comp}^* \omega_s)=\rho_s^*(\chi_s^* \omega_s)=\rho_s^* \omega=\rho_0^* \omega$, we see that $(V_{symp}, \rho_0^* \omega, \underline{\chi}_s^* J^V)$ is isomorphic to $(V_{comp}, \rho_{comp}^* \omega_s, J^V)$.
Therefore $(V_{symp}, \rho_0^* \omega, \underline{\chi}_s^* J^V)$ is a K\"ahler manifold.
\\
$(3)$ In the above construction of $\phi_s$ there is no ambiguous choice.
So $\phi_s$ is canonically defined and depends smoothly on $s \ge 0$. 
Therefore the maps $\rho_s$ and $\underline{\chi}_s$ are canonically defined and depend smoothly on $s \ge 0$.
\end{proof}

Next we construct a lift of the maps $\rho_s \colon V_{symp} \to M_{symp}$ and $\underline{\chi}_s \colon V_{symp} \to V_{comp}$ to the prequantum line bundle.

\begin{proposition}\label{identify2}
In addition to the assumption of Proposition \ref{identify1}, let $(L_{symp}^V, h^V,  \nabla^V)= \rho_0^* (L_{symp}, h, \nabla)$ and $(L_{comp}^V, \overline{\partial}^V)=\rho_{comp}^*(L_{comp}, \overline{\partial})$ be a prequantum line bundle on $(V_{symp}, \rho_0^*\omega)$ and a holomorphic line bundle on $(V_{comp}, J^V)$ respectively. 
Let $\tilde{\rho}_{comp} \colon L_{comp}^V \to L_{comp}$ be the canonical lift of the embedding $\rho_{comp} \colon V_{comp} \to M_{comp}$. 
\\
$(1)$ There exists a lift $\tilde{\rho}_s \colon L_{symp}^V \to L_{symp}$ of $\rho_s \colon V_{symp} \to M_{symp}$ such that $\tilde{\rho}_s^* (L_{symp},h, \nabla) = (L_{symp}^V,h^V, \nabla^V)$ for $s \ge 0$. 
\\
$(2)$ There exists a lift $\tilde{\underline{\chi}}_s \colon L_{symp}^V \to L_{comp}^V$ of the map $\underline{\chi}_s \colon V_{symp} \to V_{comp}$ for $s \ge 0$ such that, for $s \ge 0$, $\nabla^V$ is the Chern connection of $(L_{symp}^V, h^V, \tilde{\underline{\chi}}_s^* \overline{\partial}^V)$ and the following diagram commutes:
$$
\begin{array}{ccc}
   (L_{symp},h, \nabla) & \overset{\tilde{\chi}_s}{\longrightarrow } & (L_{comp},\bar{\partial}) \\
    \uparrow \tilde{\rho}_s &                                           & \uparrow \tilde{\rho}_{comp} \\
    (L_{symp}^V,h^V, \nabla^V) & \overset{\tilde{\underline{\chi}}_s}{\longrightarrow }        & (L_{comp}^V,\bar{\partial}^V).
\end{array}
$$
$(3)$ The maps $\tilde{\rho}_s$ and $\tilde{\underline{\chi}}_s$ are canonically defined and depend smoothly on $s \ge 0$.
\end{proposition}
\begin{proof}
We use the same notation as in the proof of Proposition \ref{identify1}. \\
$(1)$ First we show the following.
\begin{claim}\label{claim2} 
Let $R \colon V_{symp} \times [0, \infty) \to M_{symp}$ be the map defined by $R(p,s)= \rho_s(p)$. 
Then the following holds. 
$$ 
i( \frac{\partial}{\partial s})(R^* \omega)=0 ~~ \text{ on $ V_{symp} \times [0, \infty) $.} 
$$
\end{claim} 
\begin{proof}[Proof of Claim \ref{claim2}.] 
Define $\theta \colon V_{comp} \times [0, \infty) \to M_{comp}$ by $\theta(p,s)= \psi_s \circ \rho_{comp} \circ \phi_s(p)$.
Fix any $ (p_0,s_0) \in V_{symp} \times [0, \infty)$ and $v \in T_{p_0}V_{symp} \subset T_{(p_0,s_0)} \{ V_{symp} \times [0, \infty) \}$. 
We set $q_0= \underline{\chi}_0(p_0) \in V_{comp}$ and $w=(\underline{\chi}_0)_{*p_0}(v) \in T_{q_0}V_{comp} \subset T_{(q_0,s_0)} \{ V_{comp} \times [0, \infty) \}$.
Since $R(p,s)=((\chi_0)^{-1} \circ \theta)(\underline{\chi}_0(p),s)$, we have 
\begin{align*}
\{ i(\frac{\partial}{\partial s})&(R^* \omega) \}_{(p_0,s_0)}(v) 
= \omega_0(\theta_{*(q_0,s_0)}(\frac{\partial}{\partial s}), \theta_{*(q_0,s_0)}(w)).
\end{align*}
By (\ref{X}) we have
\begin{align*}
\theta_{*(q_0,s_0)}(\frac{\partial}{\partial s})
&= \frac{\partial}{\partial s}\Big\vert_{s=s_0}(\psi_s \circ \rho_{comp} \circ \phi_s)(q_0) \\
&= (X_{s_0})_{\psi_{s_0} \circ \rho_{comp} \circ \phi_{s_0} (q_0)} + (\psi_{s_0})_* (\rho_{comp})_* (\frac{\partial}{\partial s}\Big\vert_{s=s_0}\phi_s(q_0)), \\
\theta_{*(q_0,s_0)}(w)
&= (\psi_{s_0})_* (\rho_{comp})_* (\phi_{s_0})_{* q_0}(w).
\end{align*}
Thus we have
\begin{align*}
& \{ i(\frac{\partial}{\partial s})(R^* \omega) \}_{(p_0,s_0)}(v) \\ 
&= \{ \psi_{s_0}^*(i(X_{s_0}) \omega_0) \}_{\rho_{comp} \circ \phi_{s_0}(q_0)} ((\rho_{comp})_* (\phi_{s_0})_{* q_0}(w)) \\
& \hspace{20mm}+ \{ (\rho_{comp})^*(\psi_{s_0})^* \omega_0) \}((Y_{s_0})_{\phi_{s_0}(q_0)},~ (\phi_{s_0})_{* q_0}(w)) \\
&= (\eta_{s_0})_{\rho_{comp} \circ \phi_{s_0}(q_0)} ((\rho_{comp})_* (\phi_{s_0})_{* q_0}(w))+(-\rho_{comp}^* \eta_{s_0})_{\phi_{s_0}(q_0)}((\phi_{s_0})_{* q_0}(w)) \\
&=0,
\end{align*}
where we used (\ref{eta}) and (\ref{Y}) in the second equality.
This implies Claim \ref{claim2}.  \end{proof}

Consider the line bundle $(L_{symp}^\prime, h^\prime, \nabla^\prime)=R^*(L_{symp},h,\nabla)$ on $V_{symp} \times [0,\infty)$.
Let $S^\prime \subset L_{symp}^\prime$ be the unit sphere bundle and $p \colon S^\prime \to V_{symp} \times [0,\infty)$ the projection.
If we denote the connection form of $\nabla^\prime$ by $\alpha \in \Omega^1(S^\prime)$, then we have $d \alpha =p^* R^* \omega$.
If we denote the horizontal lift of $\frac{\partial}{\partial s} \in \mathcal{X}(V_{symp} \times [0,\infty))$ by $\xi \in \mathcal{X}(S^\prime)$, we have $i(\xi) \alpha =0$ and $p_*\xi = \frac{\partial}{\partial s}$.
So we have 
$$ 
L_\xi \alpha= i(\xi)d \alpha 
= i(\xi)(p^* R^* \omega) =  p^* \{ i(p_*\xi) (R^* \omega) \} = p^* \{ i(\frac{\partial}{\partial s}) (R^* \omega) \} = 0.
$$
Thus the flow defined by the vector field $\xi \in \mathcal{X}(S^\prime)$ preserves the connection $\nabla^\prime$.
So it induces a lift $\tilde{\rho}_s \colon L_{symp}^V \to L_{symp}$ of the map $\rho_s \colon V_{symp} \to M_{symp}$ such that $\tilde{\rho}_s^*(h,\nabla)=(h^V,\nabla^V)$ for $s \ge 0$. 
\newline 
$(2)$ Since $\chi_s \circ \rho_s = \rho_{comp} \circ \underline{\chi}_s$ holds, $\tilde{\underline{\chi}}_s = \tilde{\rho}_{comp}^{-1} \circ \tilde{\chi}_s \circ \tilde{\rho}_s \colon L_{symp}^V \to L_{comp}^V$ is well defined.
Since $\tilde{\underline{\chi}}_s^* \overline{\partial}^V=\tilde{\underline{\chi}}_s^*(\tilde{\rho}_{comp}^* \overline{\partial})=\tilde{\rho}_s^* (\tilde{\chi}_s^* \overline{\partial})$, we see that $(L_{symp}^V, h^V, \tilde{\underline{\chi}}_s^* \overline{\partial}^V)$ is isomorphic to $\tilde{\rho}_s^* (L_{symp},h,\tilde{\chi}_s^* \overline{\partial})$.
Since $\nabla$ is the Chern connection of $(L_{symp},h,\tilde{\chi}_s^* \overline{\partial})$, $\nabla^V=\tilde{\rho}_s^* \nabla$ is the Chern connection of $(L_{symp}^V, h^V, \tilde{\underline{\chi}}_s^* \overline{\partial}^V)$.
\\
$(3)$ This is obvious from the definition of the maps $\tilde{\rho}_s$ and $\tilde{\underline{\chi}}_s$.
 \end{proof} 
\subsection{Toric subvarieties}\label{subsec-toric-subm} 
Let $(M_{symp}, \omega)$ and $(M_{comp},J)$ be a symplectic and complex toric manifold, respectively, constructed from a Delzant polytope $\Delta$ defined by (\ref{polytope}).
In this subsection we study a (possibly singular) toric subvariety $V_{comp}$ of $(M_{comp},J)$ under the deformation of toric K\"ahler structures of the ambient toric manifold. 

Fix a $\mathbb{Z}$-basis $p_1, \dots, p_n \in (\mathfrak{t}^n)^*_\mathbb{Z}$ and its dual basis $q_1, \dots, q_n \in \mathfrak{t}^n_\mathbb{Z}$.
This induces symplectic coordinate $(x,[\theta])$ on $M_{symp}^0$ as in Subsection \ref{symp-toric-mfld} and complex coordinate $w=(w_1, \dots, w_n)$ on $M_{comp}^0$ as in Subsection \ref{comp-toric-mfld}.
Note that $M_{comp}^0$ is the $T^n_{\mathbb{C}}$-orbit through $e=(1, \dots, 1) \in M_{comp}^0$. 

\begin{proposition}\label{toric-subm}
Let $T^l_{\mathbb{C}}$ be an $l$-dimensional subtorus of $T^n_{\mathbb{C}}$.
Let $\iota^* \colon (\mathfrak{t}^n)^* \to (\mathfrak{t}^l)^*$ be the dual map of the inclusion of the Lie algebra $\iota \colon \mathfrak{t}^l \to \mathfrak{t}^n$.
Let $V_{comp} \subset M_{comp}$ be a closed $l$-dimensional (possibly singular) toric subvariety containing $e=(1, \dots, 1)$. 
The torus action on $V_{comp}$ is the restriction of the $T^l_{\mathbb{C}}$-action on $M_{comp}$ and its orbit through $e$ is open dense in $V_{comp}$.
\newline
$(1)$ Let $\chi_s \colon M_{symp} \to M_{comp}$ be the diffeomorphism defined by 
$$
g_s =g_0 + s(\underline{\nu} \circ \iota^*)\in SP(\Delta)\hspace*{10mm}\text{for $s \ge 0$,} 
$$ 
where $\underline{\nu} \colon \iota ^*(\Delta) \to \mathbb{R}$ is a smooth strictly convex function. 
Set $V_{symp}=(\chi_0)^{-1}(V_{comp})$. 
Then $\chi_0 \circ \chi_s^{-1}|_{V_{comp}} \colon V_{comp} \to V_{comp}$ is a homeomorphism for each $s \ge 0$. 
\newline
$(2)$ Let $\rho_s \colon V_{symp} \to M_{symp}$ and $\underline{\chi}_s \colon V_{symp} \to V_{comp}$ be the maps constructed in Proposition \ref{identify1}. 
Then $\rho_s=\rho_0$ and $\underline{\chi}_s= \chi_s|_{V_{symp}}$ hold for $s \ge 0$. 
Moreover, their lifts constructed in Proposition \ref{identify2} are given by $\tilde{\rho}_s=\tilde{\rho}_0 \colon L_{symp}^V \to L_{symp}$ and $\tilde{\underline{\chi}}_s=\tilde{\chi}_s|_{L_{symp}^V} \colon L_{symp}^V \to L_{comp}^V$.
\end{proposition}
\begin{proof}
$(1)$ Note that $V_{comp} \cap M_{comp}^0$ is a connected component of 
\begin{equation}\label{V-descr} 
\{ w \in M_{comp}^0~|~ \prod_{i=1}^n w_i^{\langle p, q_i \rangle}=1 ~~ \text{for all $p \in \ker \iota^* \cap (\mathfrak{t}^n)^*_{\mathbb{Z}}$} \}, 
\end{equation}
which contains $e=(1, \dots, 1)$. 
By (\ref{chi-g}) we see that $(\chi_s)^{-1}(V_{comp} \cap M_{comp}^0)$ is a connected component of 
\begin{equation}\label{V-descr2} 
\{ (x, [\theta]) \in M_{symp}^0 ~|~ e^{2 \pi \sum_{i=1}^n \langle p, q_i \rangle (\frac{\partial g_s}{\partial x_i}+ \sqrt{-1} \theta_i)} =1 ~~ \text{for all $p \in \ker \iota^* \cap (\mathfrak{t}^n)^*_{\mathbb{Z}}$} \}.
\end{equation}
On the other hand, we have $$\sum_{i=1}^n \langle p, q_i \rangle \frac{\partial g_s}{\partial x_i}= \sum_{i=1}^n \langle p, q_i \rangle \frac{\partial g_0}{\partial x_i}+ s \sum_{i=1}^n \langle p, q_i \rangle \frac{\partial (\underline{\nu} \circ \iota^*)}{\partial x_i}=\sum_{i=1}^n \langle p, q_i \rangle \frac{\partial g_0}{\partial x_i},$$
because $\sum_{i=1}^n \langle p, q_i \rangle \frac{\partial}{\partial x_i}$ is a differential in the direction of $\ker \iota^*$ for all $p \in \ker \iota^* \cap (\mathfrak{t}^n)^*_{\mathbb{Z}}$. 
Therefore we have $(\chi_s)^{-1}(V_{comp}) \subset (\chi_0)^{-1}(V_{comp}) =V_{symp}$.  
So we have an injective continuous map $\chi_0 \circ \chi_s^{-1}|_{V_{symp}} \colon V_{comp} \to V_{comp}$. 

Next we show that the map $\chi_0 \circ \chi_s^{-1}|_{V_{comp}} \colon V_{comp} \to V_{comp}$ is surjective. 
Let $V_{comp}^0$ be the $T^l_\mathbb{C}$-orbit through $e$. 
Note that $\chi_0 \circ \chi_s^{-1}|_{V_{symp}}$ is $T^l$-equivariant and injective.
If we consider the isotropy subgroup at each point, we have $\chi_0 \circ \chi_s^{-1}(V_{comp}^0) \subset V_{comp}^0$ and $\chi_0 \circ \chi_s^{-1}(V_{comp} \setminus V_{comp}^0) \subset V_{comp} \setminus V_{comp}^0$. 
Since $\chi_0 \circ \chi_s^{-1}|_{V_{comp}^0}$ is a $C^\infty$-map and its differential is an isomorphism at each point, $\chi_0 \circ \chi_s^{-1}(V_{comp}^0)$ is open in $V_{comp}^0$. 
On the other hand, since $V_{comp}$ is compact, $\chi_0 \circ \chi_s^{-1}(V_{comp})$ is compact. 
Therefore we see that $\chi_0 \circ \chi_s^{-1}|_{V_{comp}^0} \colon V_{comp}^0 \to V_{comp}^0$ is surjective. 
So we see that $\chi_0 \circ \chi_s^{-1}|_{V_{comp}} \colon V_{comp} \to V_{comp}$ is surjective. 

Since $\chi_0 \circ \chi_s^{-1}|_{V_{comp}} \colon V_{comp} \to V_{comp}$ is a bijective continuous map and $V_{comp}$ is a compact Hausdorff space, it is a homeomorphism.
\newline
$(2)$ In the proof of Proposition \ref{identify1} we constructed $\phi_s \colon V_{comp} \to V_{comp}$ by integrating the time-dependent vector field $Y_s \in \mathcal{X}(V_{comp})$ defined by 
$$
i(Y_s)(\rho_{comp}^* \omega_s) + \rho_{comp}^* \eta_s =0, ~~ \text{where $
\eta_s = \psi_s^* \{ i(X_s) \omega_0 \} \in \Omega^1(M_{comp}).$}
$$
In our situation the vector field $Y_s$ is defined on $V_{comp}^0$. 
Since $V_{comp}^0$ is non-compact, it is not obvious that $Y_s$ is integrated to define the map $\phi_s |_{V_{comp}^0} \colon V_{comp}^0 \to V_{comp}^0$. 
However, we show that this holds in our case and that $\phi_s |_{V_{comp}^0}$ extends to a homeomorphism $\phi_s \colon V_{comp} \to V_{comp}$.

In the proof of Proposition \ref{toric-subm} $(1)$, we showed that $\psi_s |_{V_{comp}^0}=\chi_0 \circ \chi_s^{-1} |_{V_{comp}^0} \colon V_{comp}^0 \to V_{comp}^0$ is a diffeomorphism.
Moreover, by (\ref{X}) the restriction $X_s |_{V_{comp}^0}$ takes its values in the tangent bundle of $V_{comp}^0$.
If we note (\ref{omega-s}), we have 
\begin{align*}\rho_{comp}^* \eta_s &= \rho_{comp}^* \psi_s^* \{ i(X_s) \omega_0 \} \\
& = \rho_{comp}^* \{ i((\psi_s^{-1})_* X_s)  \psi_s^* \omega_0 \} =  i((\psi_s^{-1})_* (X_s |_{V_{comp}^0}))  \rho_{comp}^* \omega_s.
\end{align*} 
Thus we have 
$$
0=i(Y_s)(\rho_{comp}^* \omega_s) + \rho_{comp}^* \eta_s =  i(Y_s + (\psi_s^{-1})_* (X_s |_{V_{comp}^0}))  \rho_{comp}^* \omega_s.
$$
So we have 
$$
Y_s + (\psi_s^{-1})_* (X_s |_{V_{comp}^0}) =0 \in \mathcal{X}(V_{comp}^0).
$$
For any $p \in V_{comp}^0$ we have 
$$
(Y_s)_p =- \{ (\psi_s^{-1})_* (X_s |_{V_{comp}^0}) \}_p 
=-\frac{d}{dt}\Big\vert_{t=0}\psi_s^{-1} \circ \psi_{s+t}(p)= \frac{d}{dt}\Big\vert_{t=0} \psi_{s+t}^{-1} \circ \psi_s(p).
$$
Namely, we have $(Y_s)_{\psi_s^{-1}(p)}= \frac{d}{dt}\big\vert_{t=0} \psi_{s+t}^{-1}(p)$.
Thus the vector field $Y_s$ on $V_{comp}^0$ is integrated to define $\phi_s |_{V_{comp}^0} = \psi_s^{-1}|_{V_{comp}^0} = \chi_s \circ \chi_0^{-1}|_{V_{comp}^0}$.
So $\phi_s |_{V_{comp}^0}$ is extended to a homeomorphism $\phi_s = \chi_s \circ \chi_0^{-1}|_{V_{comp}} \colon V_{comp} \to V_{comp}$. 
So we have $\underline{\chi}_s = \phi_s \circ \chi_0|_{V_{symp}} = \chi_s |_{V_{symp}}$. 
Then the rest of the statement is obvious. 
\end{proof} 

Recall that we defined a holomorphic section $\sigma^m$ of  $L_{comp}$ for $m \in \Delta \cap (\mathfrak{t}^n)^*_{\mathbb{Z}}$ by (\ref{def-sigma-m-0}). 
Then $\tilde{\chi}_s^* \sigma^m$ is a section of $L_{symp}$.
By Proposition \ref{toric-subm} $(2)$, the section $\tilde{\rho}_s^*(\tilde{\chi}_s^* \sigma^m)=\tilde{\rho}_0^*(\tilde{\chi}_s^* \sigma^m)$ of $L_{symp}^V$ can be written as $\tilde{\chi}_s^*\sigma^m |_{V_{symp}}$. 
\begin{proposition}\label{sigma-m}
In addition to the assumptions in Proposition \ref{toric-subm}, suppose that $\iota^*((\mathfrak{t}^n)^*_{\mathbb{Z}})=(\mathfrak{t}^l)^*_{\mathbb{Z}}$. 
Set $\mu_{T^l}= \iota^* \circ \mu_{T^n} \colon M_{symp} \to (\mathfrak{t}^l)^*$ and $\Delta_V = \mu_{T^l}(V_{symp})$.
\\
$(1)$ For $m, m^\prime \in \Delta \cap (\mathfrak{t}^n)^*_{\mathbb{Z}}$, $\sigma^m |_{V_{comp}} = \sigma^{m^\prime} |_{V_{comp}}$ if $\iota^* m = \iota^* m^\prime$.  
\\
$(2)$ For $p \in \mathrm{Int}\Delta_V$, $\mu_{T^l}^{-1}(p) \cap V_{symp}$ is a Bohr-Sommerfeld fiber for the prequantum line bundle $(L_{symp},h,\nabla)|_{V_{symp}}$ if and only if $p \in \mathrm{Int}\Delta_V \cap (\mathfrak{t}^l)^*_{\mathbb{Z}}$.
\\
$(3)$ Fix any $m \in \Delta \cap (\mathfrak{t}^n)^*_{\mathbb{Z}}$ with $\iota^* m \in \mathrm{Int}\Delta_V \cap (\mathfrak{t}^l)^*_{\mathbb{Z}}$. 
Let $B_{\iota^* m}$ be an open neighborhood of $\iota^* m$ in $(\mathfrak{t}^l)^*$. 
Then there exists $C_0(s) >0$, depending continuously on $s \ge 0$, such that $\lim_{s \to \infty}C_0(s)=0$ and, for arbitrary $s \ge 0$, 
$$
\Vert \tau^m_s \Vert_{C^0(M_{symp} \setminus \mu_{T^l}^{-1}(B_{\iota^* m}))} \le C_0(s)~~\text{where $\tau^m_s=\frac{\tilde{\chi}_s^* \sigma^m}
      {\Vert \tilde{\chi}_s^*\sigma^m |_{V_{symp}} \Vert_{L^1(V_{symp})}}$.}
$$
\\
$(4)$ Fix $m \in \Delta \cap (\mathfrak{t}^n)^*_{\mathbb{Z}}$ with $\iota^* m \in \mathrm{Int}\Delta_V \cap (\mathfrak{t}^l)^*_{\mathbb{Z}}$.
The section $\tau^m_s |_{V_{symp}}$ converges to a delta-function section supported on the Bohr-Sommerfeld fiber $\mu_{T^l}^{-1}(\iota^* m) \cap V_{symp} $ in the following sense:
there exist a covariantly constant section $\delta_{\iota^* m}$ of $(L_{symp},h,\nabla)|_{\mu_{T^l}^{-1}(\iota^* m)}$ and a measure $d \theta_{\iota^* m}$ on $\mu_{T^l}^{-1}(\iota^* m) \cap V_{symp}$ such that, for any smooth section $\phi$ of the dual line bundle $(L_{symp}^V)^*$, the following holds
$$
\lim_{s \to \infty} \int_{V_{symp}} \left\langle \phi, \tau^m_s |_{V_{symp}} \right\rangle 
   \frac{(\rho_0^*\omega)^l}{l !} 
= \int_{\mu_{T^l}^{-1}(\iota^* m) \cap V_{symp}} \langle \phi , \delta_{\iota^* m} \rangle d \theta_{\iota^* m}.
$$
\end{proposition}
\begin{proof}
$(1)$ By (\ref{def-sigma-m}) we have $\sigma^{m^\prime}/\sigma^m = \prod_{i=1}^n w_i^{\langle m^\prime -m, q_i \rangle}$ on $M_{comp}^0$. 
Since $m^\prime -m \in \ker \iota^*$, due to (\ref{V-descr}), we have $\sigma^{m^\prime}/\sigma^m = 1$ on $V_{comp}$. 
\\
$(2)$ Since $\iota^*((\mathfrak{t}^n)^*_{\mathbb{Z}})=(\mathfrak{t}^l)^*_{\mathbb{Z}}$, we can take $p_1^\prime, \dots, p_l^\prime \in (\mathfrak{t}^n)^*_{\mathbb{Z}}$ so that $\iota^* p_1^\prime, \dots, \iota^* p_l^\prime$ is a $\mathbb{Z}$-basis of $(\mathfrak{t}^l)^*_{\mathbb{Z}}$. 
In addition, if we fix a $\mathbb{Z}$-basis $p_{l+1}^\prime, \dots, p_n^\prime$ of $(\ker \iota^*) \cap (\mathfrak{t}^l)^*_{\mathbb{Z}}$, then $p_1^\prime, \dots, p_n^\prime$ is a $\mathbb{Z}$-basis of $(\mathfrak{t}^n)^*_{\mathbb{Z}}$. 
It induces the complex coordinate $w^\prime=(w_1^\prime, \dots, w_n^\prime)$ on $M_{comp}^0$ and the symplectic coordinate $(x^\prime, [\theta^\prime])$ on $M_{symp}^0$ as in the previous subsections.

By (\ref{V-descr}) we have $V_{comp} \cap M_{comp}^0 = \{ w^\prime \in M_{comp}^0~|~ w_{l+1}^\prime= \dots =w_n^\prime=1 \}$.
So, by (\ref{V-descr2}), we see that $\theta_{l+1}^\prime, \dots, \theta_n^\prime$ are constant on $V_{symp} \cap M_{symp}^0$. 
Moreover, by (\ref{mm}) we have $\mu_{T^l}(x^\prime, [\theta^\prime])= \sum_{i=1}^lx_i^\prime p_i^\prime$ for $(x^\prime, [\theta^\prime]) \in M_{symp}^0$. 
For each $p= \sum_{i=1}^lx_i^\prime p_i^\prime \in \mathrm{Int}\Delta_V$, since $\mu_{T^l}^{-1}(p) \cap V_{symp}$ is a single $T^l$-orbit, $x_{1}^\prime, \dots, x_n^\prime$ are also constant on $\mu_{T^l}^{-1}(p) \cap V_{symp}$. 

On the other hand, due to Lemma \ref{cansymp}, we see that $\nabla|_{M_{symp}^0}= d- 2 \pi \sqrt{-1} \sum_{i=1}^n x_i^\prime d \theta_i^\prime$ with respect to the trivialization defined by $s^0_{symp}$. 
Therefore, for a fixed $p= \sum _{i=1}^l x_i^\prime \iota^*p_i^\prime \in \mathrm{Int}\Delta_V$, the multi-valued section $\delta_p([\theta^\prime])=e^{2 \pi \sqrt{-1}\sum_{i=1}^l x_i^\prime \theta_i^\prime}s^0_{symp}$ of  $(L_{symp},h,\nabla)|_{\mu_{T^l}^{-1}(p) \cap V_{symp}}$ is covariantly constant.
Since $\delta_p$ is single-valued if and only if $p \in \mathrm{Int}\Delta_V \cap (\mathfrak{t}^l)^*_{\mathbb{Z}}$, we finish the proof. 
\\
$(3)$ The following proof is a slight modification of the argument in \cite{BFMN}.
If we write $m =\sum_{i=1}^n m_i^\prime p_i^\prime \in (\mathfrak{t}^n)^*_{\mathbb{Z}}$, due to (\ref{def-sigma-m}) and Theorem \ref{GuAb}, we have
\begin{align*}
\tilde{\chi_s}^* \sigma^m & = \tilde{\chi_s}^* \{ (\prod_{i=1}^n (w_i^\prime)^{m_i^\prime}) s_{comp}^0 \} \\
& = \lbrace \prod_{i=1}^n e^{2\pi m_i^\prime (\frac{\partial g_s}{\partial x_i^\prime}+ \sqrt{-1} \theta_i^\prime)} \rbrace e^{2\pi (g_s-\sum_{i=1}^n x_i^\prime \frac{\partial g_s}{\partial x_i^\prime} )}s_{symp}^0 \\
& =  e^{2 \pi (g_s - \sum_{i=1}^n (x_i^\prime -m_i^\prime)\frac{\partial g_s}{\partial x_i^\prime})} e^{2 \pi \sqrt{-1} (\sum_{i=1 }^n m_i^\prime \theta_i^\prime)} s_{symp}^0 \\
& =  e^{-2 \pi s \alpha_m (x^\prime)} \varsigma^m,
\end{align*}
where
\begin{align*}
& \varsigma^m(x^\prime, [\theta^\prime]) = e^{2 \pi (g_0 - \sum_{i=1}^n (x_i^\prime -m_i^\prime)\frac{\partial g_0}{\partial x_i^\prime})}e^{2 \pi \sqrt{-1} (\sum_{i=1 }^n m_i^\prime \theta_i^\prime)} s_{symp}^0(x^\prime, [\theta^\prime]), \\
& \alpha_m(x^\prime) =\sum_{i=1}^n (x_i^\prime -m_i^\prime)\frac{\partial (\underline{\nu} \circ \iota^*)}{\partial x_i^\prime}(x^\prime)-(\underline{\nu} \circ \iota^*)(x^\prime) .
\end{align*}
If we set $\underline{\alpha}_{\iota^* m}(p)=\sum_{i=1}^l (x_i^\prime -m_i^\prime)\frac{\partial \underline{\nu}}{\partial x_i^\prime}(p)-\underline{\nu}(p)$ for $p=\sum_{i=1}^l x_i^\prime \iota^* p_i^\prime \in \iota^* \Delta \subset (\mathfrak{t}^l)^*$, then we have $\alpha_m(x^\prime) =\underline{\alpha}_{\iota^* m}\circ \mu_{T^l}(x^\prime, [\theta^\prime])$.
As in the argument in Section 4 in \cite{BFMN}, we have 
\begin{align*}
\underline{\alpha}_{\iota^* m}(p)
&= \underline{\alpha}_{\iota^* m}(\iota^* m)+ \int_0^1 \frac{d}{dt} \underline{\alpha}_{\iota^* m}(\iota^* m + t(p-\iota^* m))dt \\
&= - \underline{\nu}(\iota^* m)+ \int_0^1 {^t(p-\iota^* m)}(\mathrm{Hess}_{\iota^* m + t(p-\iota^* m)}\underline{\nu})(p-\iota^* m) dt. 
\end{align*}
Since $\underline{\nu} \colon \iota ^*(\Delta) \to \mathbb{R}$ is strictly convex and $\iota^* \Delta$ is compact, if we put $\Vert p \Vert^2 = \sum_{i=1}^l (x_i^\prime)^2$ for $p=\sum_{i=1}^l x_i^\prime p_i^\prime= \iota^* x^\prime \in (\mathfrak{t}^l)^*$, there exist $C_1, C_2 >0$ such that 
$$
 - \underline{\nu}(\iota^* m)+ C_1 \Vert p-\iota^* m \Vert^2 \le \underline{\alpha}_{\iota^* m}(p) \le - \underline{\nu}(\iota^* m)+ C_2 \Vert p-\iota^* m \Vert^2~~~~~~\text{for $p \in \iota^* \Delta $.}
$$
So we have 
\begin{equation}\label{u}
e^{-s \underline{\alpha}_{\iota^* m}(p)} \le e^{s \underline{\nu}(\iota^* m)-s C_1 \Vert p-\iota^* m \Vert^2}~~~~~~\text{for $p \in \iota^* \Delta$.} 
\end{equation}
On the other hand, there exists $C_3 >0$, for sufficiently small $r >0$,  
$$
\int_{\Delta_V}e^{-s \underline{\alpha}_{\iota^* m}(p)}dp \ge \int_{B_r(\iota^* m) \cap \Delta_V}e^{s \underline{\nu}(\iota^* m)-s C_2 \Vert p-\iota^* m \Vert^2}dp  \ge C_3 r^l e^{s \underline{\nu}(\iota^* m)-s C_2 r^2}
$$
Since $\alpha_m=\underline{\alpha}_{\iota^* m}\circ \mu_{T^l}$ is a smooth function on $M_{symp}$, $\varsigma^m$ is a smooth section of $L_{symp}$.
Since $\varsigma^m$ is non-zero on $\mu_{T^l}^{-1}(\iota^* m)$ and independent of $s \ge 0$, there exists $C_4 >0$ such that 
\begin{equation}\label{l}
\Vert \tilde{\chi}_s^*\sigma^m |_{V_{symp}} \Vert_{L^1(V_{symp})} \ge C_4 r^l e^{s \underline{\nu}(\iota^* m)-s C_2 r^2}
\end{equation}
By (\ref{u}) and (\ref{l}) there exists $C_5 >0$ such that 
\begin{align*} 
\left\vert \tau^m_s (x^\prime,[\theta^\prime]) \right\vert & =
\left\vert \frac{\tilde{\chi}_s^* \sigma^m (x^\prime,[\theta^\prime])}
      {\Vert \tilde{\chi}_s^*\sigma^m |_{V_{symp}} \Vert_{L^1(V_{symp})}} \right\vert \\
& \le C_5 \frac{e^{s \underline{\nu}(\iota^* m)-s C_1 \Vert \iota^* x^\prime -\iota^* m \Vert^2}}{r^l e^{s \underline{\nu}(\iota^* m)-s C_2 r^2}} = C_5 r ^{-l} e^{-s(C_1 \Vert \iota^* x^\prime -\iota^* m \Vert^2-C_2 r^2)}.
\end{align*}
Since we can take small $r>0$ so that $C_1 \Vert p -\iota^* m \Vert^2-C_2 r^2 >0$ for any $p \in \iota^* \Delta \setminus B_{\iota^* m}$, we finish the proof. 
\\
$(4)$ By the above argument, we also have 
$$
\lim_{s \to \infty} \frac{e^{-2 \pi s \underline{\alpha}_{\iota^* m}(p)}}{ \Vert e^{-2 \pi s \underline{\alpha}_{\iota^* m}} \Vert_{L^1(\Delta_V)} } = \delta(p- \iota^* m)
$$ for $\iota^* m \in \mathrm{Int}\Delta_V \cap (\mathfrak{t}^l)^*_{\mathbb{Z}}$, where $\delta(x)$ is the Dirac delta function on $(\mathfrak{t}^l)^*$ supported at the origin.  
Moreover, the restriction $\varsigma^m|_{\mu_{T^l}^{-1}(\iota^* m) \cap V_{symp}}=c(e^{2 \pi \sqrt{-1} (\sum_{i=1 }^l m_i^\prime \theta_i^\prime)} s_{symp}^0)|_{\mu_{T^l}^{-1}(\iota^* m) \cap V_{symp}}$, where $c$ is a constant, is a covariantly constant section on $\mu_{T^l}^{-1}(\iota^* m) \cap V_{symp}$, which we denote by $\delta_{\iota^* m}(\theta^\prime)$.
So the assertion is proved easily. 
The details are the same as in \cite{BFMN}. 
\end{proof} 
\section{Proof of main result}

In this section we prove Theorem~\ref{mainth} by applying the method developed in Section \ref{sec-identify}. 
In Subsection \ref{setup} we explain how the setting of Theorem \ref{mainth} fits into the framework of Section \ref{sec-identify}. 
In Subsection \ref{complex-constr}, we construct a family of complex structures on the flag manifold, from which $(1)$, $(2)$ and $(3)$ of Theorem \ref{mainth} turn out to be obvious. 
Finally, we prove Theorem \ref{mainth} $(4)$ in Subsections \ref{est-gh} and \ref{conv}. 

\subsection{Set up}\label{setup}

In Section \ref{main} we fixed a symplectic structure $\omega_{\mathbb{P}}$ on $\mathbb{P}=\prod_{l=1}^{n-1} \mathbb{P}(\bigwedge^l \mathbb{C}^n)$. 
We denote the complex structure on $\mathbb{P}$ by $J_{\mathbb{P}}$. 
Note that $(\mathbb{P}, \omega_{\mathbb{P}}, J_{\mathbb{P}})$ is a toric K\"ahler manifold, constructed from a Delzant polytope $\Delta_{\mathbb{P}}$.  
Moreover, the toric K\"ahler manifold $(\mathbb{P}, \omega_{\mathbb{P}}, J_{\mathbb{P}})$ can be viewed as the identification of a symplectic toric manifold $(\mathbb{P}_{symp}, \omega_{\mathbb{P}})$ with a complex toric manifold $(\mathbb{P}_{comp}, J_{\mathbb{P}})$ by the diffeomorphism $\chi_0 \colon (\mathbb{P}_{symp}, \omega_{\mathbb{P}}) \to (\mathbb{P}_{comp}, J_{\mathbb{P}})$ defined by a symplectic potential $g_0 \in SP(\Delta_{\mathbb{P}})$, as in Subsection~\ref{symp-potl}.
Similarly, the Hermitian line bundle $(L^{\mathbb{P}},h^{\mathbb{P}},\nabla^{\mathbb{P}})$ can also be viewed as the identification of the prequantum line bundle $(L_{symp}^{\mathbb{P}},h^{\mathbb{P}},\nabla^{\mathbb{P}})$ on $(\mathbb{P}_{symp}, \omega_{\mathbb{P}})$ with the holomorphic line bundle $(L_{comp}^{\mathbb{P}}, \overline{\partial}^{\mathbb{P}})$ on $(\mathbb{P}_{comp}, J_{\mathbb{P}})$ via the bundle isomorphism $\tilde{\chi_0}$, which is a lift of the map $\chi_0 \colon \mathbb{P}_{symp} \to \mathbb{P}_{comp}$.

The flag manifold $(\mathbb{F}, \omega_{\mathbb{F}}, J_{\mathbb{F}})$ in Theorem \ref{mainth} can also be viewed as the identification of $(\mathbb{F}_{symp}, \omega_{\mathbb{F}})$ with $(\mathbb{F}_{comp}, J_{\mathbb{F}})$ as follows.
Let us denote the Pl\"ucker embedding by $\rho_{comp} \colon (\mathbb{F}_{comp}, J_{\mathbb{F}}) \to (\mathbb{P}_{comp}, J_{\mathbb{P}})$.
We set $\mathbb{F}_{symp}=\chi_0^{-1}(\mathbb{F}_{comp})$ and let $\rho_{symp} \colon \mathbb{F}_{symp} \to \mathbb{P}_{symp}$ be the embedding. 
Note that $\rho_{symp}^* \omega_{\mathbb{P}}=\omega_{\mathbb{F}}$.
We also set $(L_{symp}^{\mathbb{F}}, h^{\mathbb{F}}, \nabla^{\mathbb{F}})= \rho_{symp}^* (L_{symp}^{\mathbb{P}},h^{\mathbb{P}},\nabla^{\mathbb{P}})$ and $(L_{comp}^{\mathbb{F}}, \overline{\partial}^{\mathbb{F}})= \rho_{comp}^* (L_{comp}^{\mathbb{P}}, \overline{\partial}^{\mathbb{P}})$.
Then we have the following commutative diagrams:
\begin{equation}\label{initial} 
\begin{array}{ccccccc} 
(\mathbb{P}_{symp}, \omega_{\mathbb{P}}) \!\!\! & \overset{\chi_0}{\longrightarrow } \!\!\! & (\mathbb{P}_{comp}, J_{\mathbb{P}}) & \qquad & (L_{symp}^{\mathbb{P}},h^{\mathbb{P}},\nabla^{\mathbb{P}}) \!\!\! & \overset{\tilde{\chi}_0}{\longrightarrow } \!\!\! & (L_{comp}^{\mathbb{P}}, \overline{\partial}^{\mathbb{P}}) \\
\uparrow \rho_{symp} \!\!\! & \!\!\! & \uparrow \rho_{comp} & & \uparrow \tilde{\rho}_{symp} \!\!\! & \!\!\! & \uparrow \tilde{\rho}_{comp} \\
(\mathbb{F}_{symp}, \omega_{\mathbb{F}}) \!\!\! & \overset{\chi_0 |_{\mathbb{F}_{symp}}}{\longrightarrow } \!\!\! & (\mathbb{F}_{comp}, J_{\mathbb{F}}) & & (L_{symp}^{\mathbb{F}}, h^{\mathbb{F}}, \nabla^{\mathbb{F}}) \!\!\! & \overset{\tilde{\chi}_0 |_{L^\mathbb{F}_{symp}}}{\longrightarrow } \!\!\! & (L_{comp}^{\mathbb{F}}, \overline{\partial}^{\mathbb{F}}),
\end{array} 
\end{equation} 
where $\tilde{\rho}_{symp}$ and $\tilde{\rho}_{comp}$ are the natural embeddings.
 
In Subsection \ref{tg} we constructed a family of varieties $\{ Fl_n(t)=M_n(\mathbb{C})/\!/_t B \}_{t \in \mathbb{C}}$. 
We set $(V_{t, comp}, J_{V_t})=Fl_n(t)$ for $t \in [0,1]$ and denote the deformed Pl\"ucker embedding by $\rho_{t,comp} \colon V_{t, comp} \to \mathbb{P}_{comp}$, which is defined in Subsection \ref{tg}. 
Let $V_{t, symp}=\chi_0^{-1}(V_{t, comp})$ and $\rho_{t,0} \colon V_{t, symp} \to \mathbb{P}_{symp}$ be the embedding.  
We also set $(L_{symp}^{V_t}, h^{V_t}, \nabla^{V_t})= \rho_{t,0}^* (L_{symp}^{\mathbb{P}},h^{\mathbb{P}},\nabla^{\mathbb{P}})$ and $(L_{comp}^{V_t}, \overline{\partial}^{V_t})= \rho_{t, comp}^* (L_{comp}^{\mathbb{P}}, \overline{\partial}^{\mathbb{P}})$. 
Then we have a commutative diagram, which is the case $s=0$ in the diagram (\ref{deform}) below.
Note that $V_{1,comp}=\mathbb{F}_{comp}$ and $V_{1,symp}=\mathbb{F}_{symp}$ and that $V_{0,comp}$ is the Gelfand-Cetlin toric variety $Fl_n(0) \subset \mathbb{P}$.  
Thus the above family $\{ V_{t,comp} \}_{t \in [0,1]}$ connects the flag manifold $\mathbb{F}_{comp}$ with the Gelfand-Cetlin toric variety $Fl_n(0)$.

For any $t \in [0,1]$, fix a path $\gamma_t \colon [0,1] \to \mathbb{C}^{n-1}$, which is given by straight lines connecting the points:  
\begin{align*} 
\gamma_t (0)=(1, \dots, 1) \to (1, \dots, 1, t) \to (1, \dots, 1, t, t) \to \dots \to (t, \dots, t)=\gamma_t (1).
\end{align*}
Recall that we constructed a family of varieties $\{ Fl_n(\tau)=M_n(\mathbb{C})/\!/_\tau B \}$ for $\tau \in (\mathbb{C}^\times)^{n-1}$ in Subsection~\ref{stage} and that we also constructed a degeneration in stages by extending the family.
The path $\gamma_t$ is an approximation to the path $\gamma_0$ for the degeneration in stages. 
Note that $Fl_n(\gamma_t (1))=(V_{t, comp},J_{V_t})$ for $t \in [0,1]$.
Due to Propositions \ref{ghf} and \ref{liftghf}, the gradient-Hamiltonian flow along the path $\gamma_t$ for $t \in (0,1]$ gives a symplectic diffeomorphism and its lift as in the following diagram: 
\begin{equation}\label{ghflow}
\begin{array}{ccc}
(L_{symp}^{\mathbb{F}}, h^{\mathbb{F}}, \nabla^{\mathbb{F}}) & \overset{\tilde{\Psi}_t}{\longrightarrow } & (L_{symp}^{V_t}, h^{V_t}, \nabla^{V_t}) \\ 
\downarrow  & & \downarrow  \\
(\mathbb{F}_{symp}, \rho_{symp}^* \omega_{\mathbb{P}}) & \overset{\Psi_t}{\longrightarrow } & (V_{t,symp}, \rho_{t,0}^* \omega_{\mathbb{P}}).
\end{array}
\end{equation}

We can also extend $\Psi_t \colon \mathbb{F}_{symp} \to V_{t,symp}$ in (\ref{ghflow}) to the case $t=0$ if we restrict its domain to an open dense subset $\mathbb{F}_{symp}^\circ$ of $\mathbb{F}_{symp}$. 
It is already given by (\ref{psi-0}). 
Using the notation in this section, it should be written as $\Psi_0 \colon \mathbb{F}_{symp}^\circ \to V_{0,symp}^\circ$. 
We also have its lift to the prequantum line bundle.
Thus we have the following:
\begin{equation}\label{psi-0-2}
\begin{array}{ccc} 
(L_{symp}^{\mathbb{F}}, h^{\mathbb{F}}, \nabla^{\mathbb{F}})|_{\mathbb{F}_{symp}^\circ} & \overset{\tilde{\Psi}_0}{\longrightarrow } & (L_{symp}^{V_0}, h^{V_0}, \nabla^{V_0})|_{V_{0,symp}^\circ} \\ 
\downarrow  & & \downarrow  \\
(\mathbb{F}_{symp}^\circ, \rho_{symp}^* \omega_{\mathbb{P}}) & \overset{\Psi_0}{\longrightarrow } & (V_{0,symp}^\circ, \rho_{0,0}^* \omega_{\mathbb{P}}).
\end{array}
\end{equation} 

\subsection{Construction of a family of complex structures}\label{complex-constr}

On $(\mathbb{P}, \omega_{\mathbb{P}}, J_{\mathbb{P}})$, a $(\frac{1}{2} \dim_\mathbb{R} \mathbb{P})$-dimensional torus $T_{\mathbb{P}}$ acts with an open dense orbit.  
On the Gelfand-Cetlin toric variety $Fl_n(0)=V_{0,comp} \subset \mathbb{P}$,  a $(\frac{1}{2} \dim_\mathbb{R} \mathbb{F})$-dimensional torus $T_{GC}$ acts with an open dense subset, as explained in Subsection \ref{stage}. 
There is an injective homomorphism $\tilde{\iota}_{GC} \colon T_{GC} \to T_{\mathbb{P}}$ such that the embedding $\rho_{0,comp} \colon V_{0,comp} \to \mathbb{P}_{comp}$ is equivariant. 
It is described explicitly in Section 6 in \cite{NNU}. 
Let $\iota_{GC}^* \colon \mathfrak{t}_{\mathbb{P}}^* \to \mathfrak{t}_{GC}^*$ be the dual map of the inclusion of the Lie algebras $\iota_{GC} \colon \mathfrak{t}_{GC} \to \mathfrak{t}_{\mathbb{P}}$.
From the description of the map $\tilde{\iota}_{GC} \colon T_{GC} \to T_{\mathbb{P}}$ given in \cite{NNU} we see that $\iota_{GC}^*((\mathfrak{t}_{\mathbb{P}})^*_{\mathbb{Z}})=(\mathfrak{t}_{GC})^*_{\mathbb{Z}}$.

Fix a strictly convex function $\underline{\nu} \colon \mathfrak{t}_{GC}^* \to \mathbb{R}$ and set $\nu = \underline{\nu} \circ \iota_{GC}^* \colon \mathfrak{t}_{\mathbb{P}}^* \to \mathbb{R}$.
Let us consider the diffeomorphism $\chi_s \colon (\mathbb{P}_{symp}, \omega_{\mathbb{P}}) \to (\mathbb{P}_{comp}, J_{\mathbb{P}})$ defined by $g_s=g_0 + s \nu \in SP(\Delta_{\mathbb{P}})$.
Due to Propositions \ref{identify1} and \ref{identify2}, we have the following commutative diagrams:
\begin{equation}\label{deform}
\begin{array}{ccccccc}
(\mathbb{P}_{symp}, \omega_{\mathbb{P}}) \!\!\! & \overset{\chi_s}{\longrightarrow } \!\!\! & (\mathbb{P}_{comp}, J_{\mathbb{P}}) & \!\!\!\! & (L_{symp}^{\mathbb{P}},h^{\mathbb{P}},\nabla^{\mathbb{P}}) \!\!\! & \overset{\tilde{\underline{\chi}}_s}{\longrightarrow } \!\!\! & (L_{comp}^{\mathbb{P}}, \overline{\partial}^{\mathbb{P}}) \\ 
\uparrow \rho_{t,s} \!\!\! &  \!\!\! & \uparrow \rho_{t,comp} & \!\!\!\! & \uparrow \tilde{\rho}_{t,s} \!\!\! &   \!\!\! & \uparrow \tilde{\rho}_{t,comp} \\
(V_{t,symp}, \rho_{t,0}^* \omega_{\mathbb{P}})  \!\!\! & \overset{\underline{\chi}_{t,s}}{\longrightarrow }         \!\!\! & (V_{t,comp}, J_{V_t}) & \!\!\!\! & (L_{symp}^{V_t}, h^{V_t}, \nabla^{V_t}) \!\!\! & \overset{\tilde{\underline{\chi}}_{t,s}}{\longrightarrow } \!\!\! & (L_{comp}^{V_t},\overline{\partial}^{V_t}),
\end{array}
\end{equation}
where $\underline{\chi}_{t,0}=\chi_0 |_{V_{t,symp}}$ and $\tilde{\underline{\chi}}_{t,0}=\tilde{\chi}_0|_{L_{t,symp}^V}$.
Note that $\rho_{t,s}^* \omega_{\mathbb{P}}=\rho_{t,0}^* \omega_{\mathbb{P}} \in \Omega^2(V_{t,symp})$. 

In the case $(t,s)=(1,0)$ the diagrams (\ref{deform}) are the same as the diagrams (\ref{initial}). 
In the case $t=0$ the diagrams (\ref{deform}) describe the deformation of toric K\"ahler structures of the Gelfand-Cetlin toric variety $V_{0,comp}$. 
The defining equation of the image of the embedding $\rho_{0,comp} \colon V_{0,comp} \to \mathbb{P}_{comp}$ is given by the equations $(7)$ in \cite{NNU}.  
From this description we see that the image $\rho_{0,comp}(V_{0,comp})$ contains the point $(1,\dots,1) \in \mathbb{P}$ in the notation in Proposition \ref{toric-subm}.  
So Proposition \ref{sigma-m} can be applied to our case. 
Therefore the holomorphic sections on $V_{0,comp}$ converge to delta-function sections supported on Bohr-Sommerfeld fibers as $s$ goes to infinity. 
Therefore the holomorphic sections on $V_{t,comp}$ are close to delta-function sections when $t$ and $s$ go to zero and infinity, respectively, at the same time. 
So we make $t$ a function of $s$ as follows: 
Let $t \colon [0, \infty) \to \mathbb{R}$ be a monotone decreasing $C^\infty$-function with $t(0)=1$ and $\lim_{s \to \infty}t(s)=0$. 
(In fact, $t(s)$ should be required to satisfy additional conditions, which will be discussed in Lemma \ref{2nd-est} below.)

We define a complex structure $J_s$ on $(\mathbb{F}_{symp}, \rho_{symp}^* \omega_{\mathbb{P}})$ as the pull back of $J_{V_{t(s)}}$ by the following composition of diffeomorphisms, which appeared in the diagrams (\ref{ghflow}) and (\ref{deform}):
$$
(\mathbb{F}_{symp}, \rho_{symp}^* \omega_{\mathbb{P}})  \overset{\Psi_t}{\longrightarrow }  (V_{t,symp}, \rho_{t,0}^* \omega_{\mathbb{P}})  \overset{\underline{\chi}_{t,s}}{\longrightarrow } (V_{t,comp}, J_{V_t}).
$$
Namely, a family of complex structures $\{ J_s \}_{s \in [0, \infty)}$ on $(\mathbb{F}_{symp}, \rho_{symp}^* \omega_{\mathbb{P}})$ is defined by
\begin{equation}\label{js}
J_s= (\underline{\chi}_{t (s),s} \circ \Psi_{t (s)})^* J_{V_{t(s)}}. 
\end{equation}
Then $(1)$ and $(2)$ of Theorem \ref{mainth} follow from the construction of $\{ J_s \}_{s \in [0, \infty)}$.
By Proposition \ref{identify1} $(2)$, $(V_{t,symp}, \rho_{t,0}^* \omega_{\mathbb{P}}, \underline{\chi}_{t,s}^* J_{V_t})$ is a K\"ahler manifold.
Moreover, $(V_{t(s),symp}, \rho_{t(s),0}^* \omega_{\mathbb{P}}, \underline{\chi}_{t(s),s}^* J_{V_{t(s)}})$ is isomorphic to $(\mathbb{F}_{symp}, \rho_{symp}^*\omega_\mathbb{P},J_s )$ as a K\"ahler manifold. 
So Theorem \ref{mainth} $(3)$ follows as well. 
Thus, for any $s \in [0,\infty)$, $J_s$ induces the holomorphic structure $\overline{\partial}^s$ of the Hermitian line bundle $(L^\mathbb{F}_{symp},h^\mathbb{F},\nabla^\mathbb{F})$.  
Note that the map $ \tilde{\underline{\chi}}_{t(s),s} \circ \tilde{\Psi}_{t(s)} \colon (L_{symp}^\mathbb{F}, \overline{\partial}^s) \to (L_{comp}^{V_{t(s)}},\overline{\partial}^{V_{t(s)}})$ is an isomorphism of holomorphic line bundles.

To prove Theorem \ref{mainth} $(4)$, we have to construct a basis $\{ \sigma^m_s ~|~ m \in \Delta_{GC} \cap (\mathfrak{t}_{GC})^*_\mathbb{Z} \}$ of the space of holomorphic sections $H^0(L^\mathbb{F}_{symp}, \overline{\partial}^s)$.  
Recall that the Gelfand-Cetlin polytope $\Delta_{GC}$ is considered as a subset of $\mathfrak{t}_{GC}^*$ as explained in Subsection \ref{stage}.  
Since $\iota_{GC}^* ( (\mathfrak{t}_{\mathbb{P}})^*_\mathbb{Z})= (\mathfrak{t}_{GC})^*_\mathbb{Z}$, for each $m \in \Delta_{GC} \cap (\mathfrak{t}_{GC})^*_\mathbb{Z}$, we can choose $\tilde{m} \in \Delta_{\mathbb{P}} \cap (\mathfrak{t}_{\mathbb{P}})^*_\mathbb{Z}$ such that $\iota^*(\tilde{m})=m$. 
Let $\sigma^{\tilde{m}}$ be the holomorphic section of $(L_{comp}^{\mathbb{P}}, \overline{\partial}^{\mathbb{P}})$ defined by (\ref{def-sigma-m-0}).  
Then $\{ (\tilde{\rho}_{0,comp})^* \sigma_{\tilde{m}} ~|~ m \in \Delta_{GC} \cap (\mathfrak{t}_{GC})^*_\mathbb{Z} \}$ is a basis of the space of holomorphic sections $H^0(L_{comp}^{V_0}, \overline{\partial}^{V_0})$.
So there exists $s_0 >0$ such that, for any $s \ge s_0$, $\{ (\tilde{\rho}_{t (s),comp})^* \sigma^{\tilde{m}} ~|~ m \in \Delta_{GC} \cap (\mathfrak{t}_{GC})^*_\mathbb{Z} \}$ turns out to be a basis of the space of holomorphic sections $H^0(L_{comp}^{V_{t (s)}}, \overline{\partial}^{V_{t (s)}})$. 
So we define, for $s \ge s_0$,
\begin{equation}\label{sigma-ms}
\sigma^m_s= (\underline{\tilde{\chi}}_{t (s),s} \circ\tilde{\Psi}_{t (s)})^* ( (\tilde{\rho}_{t (s),comp})^* \sigma^{\tilde{m}}) ~~\text{for $m \in \Delta_{GC} \cap (\mathfrak{t}_{GC})^*_\mathbb{Z}$.}
\end{equation}
Since all $(V_{t (s), comp}, J_{V_{t (s)}})$ and all $(L_{comp}^{V_{t (s)}}, \overline{\partial}^{V_{t (s)}})$ are isomorphic for $s \ge 0$ as complex manifolds and holomorphic line bundles respectively, we can extend a basis $\{ \sigma^m_s ~|~ m \in \Delta_{GC} \cap (\mathfrak{t}_{GC})^*_\mathbb{Z} \}$ of the space of holomorphic sections $H^0(L^\mathbb{F}_{symp}, \overline{\partial}^s)$ for all $s \in [0,s_0]$, which depends continuously on $s$.  
Thus we have defined the basis $\{ \sigma^m_s ~|~ m \in \Delta_{GC} \cap (\mathfrak{t}_{GC})^*_\mathbb{Z} \}$ of the space of holomorphic sections $H^0(L^\mathbb{F}_{symp}, \overline{\partial}^s)$ for all $s \ge 0$. 

\subsection{Another gradient-Hamiltonian flow}\label{est-gh}
To prove that the holomorphic sections defined by (\ref{sigma-ms}) converge to delta-function sections, we introduce another gradient-Hamiltonian flow.

Let us consider the family of varieties $f \colon (M_n(\mathbb{C})\times \mathbb{C})/\!/B \to \mathbb{C}$ constructed in Subsection \ref{tg}. 
Put the standard K\"ahler metric on $\mathbb{C}$. 
Consider the map $F \colon (M_n(\mathbb{C})\times \mathbb{C})/\!/B \to \mathbb{P}_{symp} \times \mathbb{C}$ given by $F(x)=(\rho_{t,0}(x),t)$ if $x \in V_{t,symp}=f^{-1}(t)$.
We put the K\"ahler metric on the smooth part of $(M_n(\mathbb{C})\times \mathbb{C})/\!/B$ by pulling back the metric on $\mathbb{P}_{symp} \times \mathbb{C}$ by the map $F$.
Consider the gradient-Hamiltonian flow along the straight-line path from $1$ to $0$ in $\mathbb{C}$.  
Since $V_{t, symp}$ are smooth manifolds for all $t \in (0, 1]$, by Lemma~\ref{gr-h-reg-pts} the vector field and thus the flow are defined on all of $V_{t, symp}$ for each $t \in (0, 1]$, and also on $V_{0, symp}^\circ$, where $V_{0, symp}^\circ$ is the same as in (\ref{psi-0-2}). 
Let $V_{t, symp}^\circ \subset V_{t, symp}$ denote the image  of $V_{0, symp}^\circ$ under the reverse flow.
Then, due to Propositions \ref{ghf} and \ref{liftghf}, we have the following symplectic diffeomorphism and its lift defined by the gradient-Hamiltonian flow for $t \in [0,1]$:
$$
\begin{array}{ccc}
(L_{symp}^{V_t}, h^{V_t}, \nabla^{V_t})|_{V_{t,symp}^\circ} & \overset{\tilde{\Phi}_t}{\longrightarrow } & (L_{symp}^{V_0}, h^{V_0}, \nabla^{V_0})|_{V_{0,symp}^\circ} \\ 
\downarrow  & & \downarrow  \\
(V_{t,symp}^\circ, \rho_{t,0}^* \omega_{\mathbb{P}}) & \overset{\Phi_t}{\longrightarrow } & (V_{0,symp}^\circ, \rho_{0,0}^* \omega_{\mathbb{P}}).
\end{array}
$$

Let $\mu_{T_\mathbb{P}} \colon \mathbb{P}_{symp} \to \mathfrak{t}_\mathbb{P}^*$ be the moment map for the $T_\mathbb{P}$-action on $(\mathbb{P}_{symp},\omega_\mathbb{P})$. 
Set $\mu_{T_{GC}}= \iota_{GC}^* \circ \mu_{T_\mathbb{P}} \colon \mathbb{P}_{symp} \to \mathfrak{t}_{GC}^*$.
Fix an open set $B \subset \mathrm{Int}\Delta_{GC}$ such that $\mathrm{Int}\Delta_{GC} \cap (\mathfrak{t}_{GC})^*_{\mathbb{Z}} \subset B$ and $\overline{B} \subset \mathrm{Int}\Delta_{GC}$.
Set $U_0 = \mu_{T_{GC}}^{-1}(B) \cap V_{0,symp} \subset V_{0,symp}^\circ$. 
Here we are considering $V_{0,symp}$ as a subset of $\mathbb{P}_{symp}$ by the embedding $\rho_{0,0}$. 
Moreover, we set $U_t=\Phi_t^{-1}(U_0)$ and $U_t^c= V_{t,symp} \setminus U_t$.
We denote the closure of $U_t$ in $V_{t, symp}$ by $\overline{U_t}$.
Note that $\overline{U_t}$ and $U_t^c$ are compact. 

Let $d_\mathbb{P}( ~~, ~~)$ be the distance on $\mathbb{P}$.
Then we have the following.
\begin{lemma}\label{phi-approx}
For an arbitrary $\epsilon > 0$, there exists $t_1>0$ such that $d_\mathbb{P}( \rho_{t,0}(x), \rho_{0,0}(\Phi_t(x))) < \epsilon$ for any $0 \le t \le t_1$ and $x \in \overline{U_t} \subset V_{t,symp}$.
\end{lemma}
\begin{proof}
Fix an arbitrary small $t_1^\prime>0$.
Then $\overline{U_t}$ consists of regular points of $f \colon (M_n(\mathbb{C}) \times \mathbb{C})/\!/B \to \mathbb{C}$ for any $0 \le t \le t_1^\prime$ and $\bigcup_{0 \le t \le t_1}\overline{U_t}$ is compact.
As noted in Lemma~\ref{gr-h-reg-pts}, $\lvert \mathrm{grad}(\Re f) \rvert$ is non-zero at regular points of $f$.  
Therefore there exists a $c>0$ such that $\lvert \mathrm{grad}(\Re f) \rvert \geq c$ on $\overline{U_t}$, for every $t \in [0,t_1^\prime]$.
Thus the gradient-Hamiltonian vector field $Z$ satisfies $\vert Z \vert \leq \frac{1}{c}$ on $\overline{U_t}$ for $t \in [0,t_1^\prime]$.
Since $\Phi_t$ is the flow of $Z$ over a ``time'' $t$, we finish the proof. 
 \end{proof}

Similarly we  have the following.
\begin{lemma}\label{psi-approx}
For an arbitrary  $\epsilon > 0$, there exists $t_2>0$ such that $d_\mathbb{P}(\rho_{1,0}(\Psi_t^{-1}( x)), \rho_{1,0}(\Psi_0^{-1} \circ \Phi_t(x))) < \epsilon$ for all $0 \le t \le t_2$ and $x \in \overline{U_t} \subset V_{t,symp}$, where $\Psi_t$ is the map in $(\ref{ghflow})$ or $(\ref{psi-0-2})$.
\end{lemma}
\begin{proof}
This follows from ``smoothness in initial conditions" results in the theory of differential
equations. Because the path $\gamma_t$ is close to the path $\gamma_0$  considered in Subsection \ref{setup} for small $t>0$, the resulting diffeomorphisms $\Psi_t$ and $\Psi_0$ are very close. 
Combining with Lemma \ref{phi-approx}, we finish the proof.
\end{proof}
\subsection{Convergence to delta-function sections}\label{conv}

For an $m \in \mathrm{Int}\Delta_{GC} \cap (\mathfrak{t}_{GC})^*_\mathbb{Z}$ we have chosen $\tilde{m} \in \Delta_{\mathbb{P}} \cap (\mathfrak{t}_{\mathbb{P}})^*_\mathbb{Z}$ such that $\iota^*(\tilde{m})=m$ and defined the holomorphic section $\sigma^m_s$ by (\ref{sigma-ms}). 
From now on we prove that, if we choose $t(s)$ appropriately for $s \ge 0$, the section $\frac{\sigma^m_s}{\Vert \sigma^m_s \Vert_{L^1(\mathbb{F}_{symp})}}$ for converges to a delta-function section supported on the Bohr-Sommerfeld fiber $\mu_{GC}^{-1}(m) $ as $s$ goes to infinity.
Set, for $0 \le t \le 1$, $s >> 0$,
$$
\tau^m_{t,s} = \frac{\tilde{\underline{\chi}}_{t,s}^* (\tilde{\rho}_{t,comp}^* \sigma^{\tilde{m}})}{\Vert \tilde{\underline{\chi}}_{t,s}^* (\tilde{\rho}_{t,comp}^* \sigma^{\tilde{m}}) \Vert_{L^1(V_{t,symp})}} 
\in H^0(L_{symp}^{V_t},\tilde{\underline{\chi}}_{t,s}^* \overline{\partial}^{V_t}).
$$
Since
$$
\tilde{\Psi}_{t}^* \tau^m_{t,s}
=\frac{\tilde{\Psi}_{t}^* \tilde{\underline{\chi}}_{t,s}^* (\tilde{\rho}_{t,comp}^* \sigma^{\tilde{m}})}{\Vert \tilde{\underline{\chi}}_{t,s}^* (\tilde{\rho}_{t,comp}^* \sigma^{\tilde{m}}) \Vert_{L^1(V_{t,symp})}}
=\frac{\tilde{\Psi}_{t}^* \tilde{\underline{\chi}}_{t,s}^* (\tilde{\rho}_{t,comp}^* \sigma^{\tilde{m}})}{\Vert \tilde{\Psi}_{t}^* \tilde{\underline{\chi}}_{t,s}^* (\tilde{\rho}_{t,comp}^* \sigma^{\tilde{m}}) \Vert_{L^1(\mathbb{F}_{symp})}}, 
$$
we have $\tilde{\Psi}_{t(s)}^* \tau^m_{t(s),s}= \frac{\sigma^m_s}{\Vert \sigma^m_s \Vert_{L^1(\mathbb{F}_{symp})}}$, where $t(s)$ will be defined in Lemma \ref{2nd-est} below.

For a section $\phi \in \Gamma((L_{symp}^{\mathbb{F}})^*)$, we denote the push-forward of $\phi$ with respect to the map $\tilde{\Psi}_t$ by $\tilde{\Psi}_{t*}\phi$, which is a section of the line bundle $(L_{symp}^{V_t})^*$ for $t>0$ or a section of $(L_{symp}^{V_0})^*$ restricted to some open dense subset of $V_{0,symp}$ for $t=0$. 
In what follows, we omit the notation for the volume form when integrating on $\mathbb{F}_{symp}$ or $V_{t, symp}$, since it is preserved by the maps $\Psi_t$ and $\Phi_t$.
First we have the following: 
\begin{lemma}\label{bscg}
$(1)$ For $m \in \mathrm{Int}\Delta_{GC}$, $\mu_{T_{GC}^{-1}}(m) \cap V_{0,symp}$ is a Bohr-Sommerfeld fiber for $(L_{symp}^{V_0},h^{V_0},\nabla^{V_0})$ if and only if $m \in (\mathfrak{t}_{GC})^*_{\mathbb{Z}}$. \\
$(2)$ For $m \in \mathrm{Int}\Delta_{GC} \cap (\mathfrak{t}_{GC})^*_{\mathbb{Z}}$, there exist a covariantly constant section $\delta_m$ of $(L_{symp}^{V_0},h^{V_0},\nabla^{V_0})|_{\mu_{T_{GC}}^{-1}(m) \cap V_{0,symp}}$ and a measure $d \theta_m$ on $\mu_{T_{GC}}^{-1}(m) \cap V_{0,symp}$ which satisfy the following: 
for any $\phi \in \Gamma((L_{symp}^{\mathbb{F}})^*)$, there exists $C_1(s, \phi)>0$ for $s \ge 0$, such that $\lim_{s \to \infty}C_1(s, \phi)=0$ and 
\begin{equation}\label{2}
\left\lvert
\int_{V_{0,symp}} \langle \tilde{\Psi}_{0*} \phi, \tau^m_{0,s}\rangle
- \int_{\mu_{T_{GC}}^{-1}(m) \cap V_{0,symp}} \langle \tilde{\Psi}_{0 *}\phi , \delta_{m} \rangle \, d\theta_m
\right\rvert \le C_1(s, \phi).
\end{equation}
\end{lemma}
\begin{proof}
$(1)$ follows from Proposition \ref{sigma-m} $(2)$. \\
$(2)$ follows from Proposition \ref{sigma-m} $(4)$.
Since the number of points in $\mathrm{Int}\Delta_{GC} \cap (\mathfrak{t}_{GC})^*_\mathbb{Z}$ is finite, we can choose $C_1(s, \phi)$ independently of $m \in \mathrm{Int}\Delta_{GC} \cap (\mathfrak{t}_{GC})^*_\mathbb{Z}$.
\end{proof}
We take $U_t \subset V_{t,symp}$ as in Subsection \ref{est-gh}. 
Then we have the following.
\begin{lemma}\label{1st-est}
For each section $\phi \in \Gamma((L_{symp}^{\mathbb{F}})^*)$, the following holds.
\begin{align*}
& \left\lvert \int_{\mathbb{F}_{symp}} \langle \phi, \tilde{\Psi}_t^*\tau^m_{t,s} \rangle  
- \int_{\mu_{T_{GC}}^{-1}(m) \cap V_{0,symp}} \langle \tilde{\Psi}_{0 *}\phi , \delta_{m} \rangle \, d\theta_m \right\rvert\\
& \le C_1(s, \phi) + \mathrm{vol}(\mathbb{F}_{symp}) \Vert \phi \Vert_{C^0(\mathbb{F}_{symp})} (\Vert \tau^m_{t,s} \Vert_{C^0(U_t^c)} + \Vert \tilde{\Phi}_t^* \tau^m_{0,s} \Vert_{C^0(U_t^c)}) \\
& \hspace*{3mm} + \mathrm{vol}(\mathbb{F}_{symp}) \Vert \phi \Vert_{C^0(\mathbb{F}_{symp})}\lVert \tau^m_{t,s} - \tilde{\Phi}_t^* \tau^m_{0,s} \rVert_{C^0(U_t )}+ \lVert \tilde{\Psi}_{t *} \phi - \tilde{\Phi}_t^* \tilde{\Psi}_{0 *} \phi \rVert_{C^0(U_t)}.  
\end{align*}
\end{lemma}
\begin{proof}
Fix arbitrary $\phi \in \Gamma((L_{symp}^{\mathbb{F}})^*)$. 
Then we have:
\begin{align}\label{1}
& \left\lvert \int_{\mathbb{F}_{symp}} \langle \phi, \tilde{\Psi}_t^*\tau^m_{t,s} \rangle  
- \int_{\mu_{T_{GC}}^{-1}(m) \cap V_{0,symp}} \langle \tilde{\Psi}_{0 *}\phi , \delta_{m} \rangle \, d\theta_m \right\rvert
\\
& =\left\lvert \int_{V_{t,symp}} \langle \tilde{\Psi}_{t *} \phi , \tau^m_{t,s} \rangle 
- \int_{\mu_{T_{GC}}^{-1}(m) \cap V_{0,symp}} \langle \tilde{\Psi}_{0 *}\phi , \delta_{m} \rangle \, d\theta_m \right\rvert
\nonumber \\
& \le
 \left\lvert \int_{V_{t,symp}} \langle \tilde{\Psi}_{t *} \phi , \tau^m_{t,s} \rangle 
 - \int_{V_{0,symp}} \langle \tilde{\Psi}_{0 *} \phi, \tau^m_{0,s} \rangle \right\rvert 
\nonumber \\
& \hspace*{10mm}  + \left\lvert \int_{V_{0,symp}} \langle \tilde{\Psi}_{0 *}\phi, \tau^m_{0,s} \rangle -
\int_{\mu_{T_{GC}}^{-1}(m) \cap V_{0,symp}} \langle \tilde{\Psi}_{0 *}\phi , \delta_{m} \rangle \, d\theta_m \right\rvert \nonumber 
\end{align} 
The second term on the right hand side of (\ref{1}) is estimated by (\ref{2}).
Next we estimate the first term on the right hand side of (\ref{1}).
\begin{align}\label{3}
& \left\lvert \int_{V_{t,symp}} \langle \tilde{\Psi}_{t *} \phi , \tau^m_{t,s} \rangle 
 - \int_{V_{0,symp}} \langle \tilde{\Psi}_{0 *} \phi,  \tau^m_{0,s} \rangle \right\rvert 
\\ 
& = \left\lvert \int_{V_{t,symp}} \langle \tilde{\Psi}_{t *} \phi , \tau^m_{t,s} \rangle 
 - \int_{V_{t,symp}} \langle \tilde{\Phi}_t^* \tilde{\Psi}_{0 *} \phi, \tilde{\Phi}_t^* \tau^m_{0,s} \rangle \right\rvert 
\nonumber \\
& \le \left\lvert \int_{U_t}
\langle \tilde{\Psi}_{t *} \phi, \tau^m_{t,s} \rangle 
- \langle \tilde{\Phi}_t^* \tilde{\Psi}_{0 *} \phi, \tilde{\Phi}_t^* \tau^m_{0,s} \rangle \right\rvert
\! + \! \left\lvert\int_{U_t^c} \langle \tilde{\Psi}_{t\ast} \phi, \tau^m_{t,s} \rangle 
- \langle \tilde{\Phi}_t^* \tilde{\Psi}_{0 \ast} \phi, \tilde{\Phi}_t^* \tau^m_{0,s} \rangle 
\right\rvert 
\nonumber \\
& \le \left\lvert \int_{U_t} 
\langle \tilde{\Psi}_{t *} \phi, \tau^m_{t,s} \rangle 
- \langle \tilde{\Phi}_t^* \tilde{\Psi}_{0 *} \phi, \tilde{\Phi}_t^* \tau^m_{0,s} \rangle \right\rvert
\nonumber \\
& \hspace*{20mm} +  \text{vol}(\mathbb{F}_{symp}) \Vert \phi \Vert_{C^0(\mathbb{F}_{symp})} (\Vert \tau^m_{t,s} \Vert_{C^0(U_t^c)} + \Vert \tilde{\Phi}_t^* \tau^m_{0,s} \Vert_{C^0(U_t^c)}). \nonumber 
\end{align}

Finally we estimate the first term on the right hand side of (\ref{3}).
If we note that 
\begin{align*}
& \int_{U_t} \lvert \tilde{\Psi}_{t *} \phi \rvert \le \mathrm{vol}(U_t) \lVert \tilde{\Psi}_{t *} \phi \rVert_{C^0(U_t)}  \le \mathrm{vol}(\mathbb{F}_{symp}) \lVert \phi \rVert_{C^0(\mathbb{F}_{symp})},
\\
& \int_{U_t} \lvert \tilde{\Phi}_t^* \tau^m_{0,s}\rvert = \int_{U_0} \lvert \tau^m_{0,s}\rvert \le \int_{U_0} \frac{\lvert \tilde{\underline{\chi}}_{0,s}^* (\tilde{\rho}_{0,comp}^* \sigma^{\tilde{m}}) \rvert}{\Vert \tilde{\underline{\chi}}_{0,s}^* (\tilde{\rho}_{0,comp}^* \sigma^{\tilde{m}}) \Vert_{L^1(V_{0,symp})}} \le 1, 
\end{align*}
then we have
\begin{align}\label{4}
& \left\lvert \int_{U_t} 
\langle \tilde{\Psi}_{t *} \phi, \tau^m_{t,s} \rangle 
- \langle \tilde{\Phi}_t^* \tilde{\Psi}_{0 *} \phi, \tilde{\Phi}_t^* \tau^m_{0,s} \rangle \right\rvert
\\
&\hspace*{0mm}  \leq \left\lvert \int_{U_t} 
\langle \tilde{\Psi}_{t *} \phi , \tau^m_{t, s} - \tilde{\Phi}_t^* \tau^m_{0,s} \rangle \right \rvert 
+ \left \lvert \int_{U_t} 
\langle \tilde{\Psi}_{t *} \phi   -  \tilde{\Phi}_t^* \tilde{\Psi}_{0 *} \phi, \tilde{\Phi}_t^* \tau^m_{0,s} \rangle \right\rvert 
\nonumber \\
& \hspace*{0mm} \leq  \lVert \tau^m_{t,s} - \tilde{\Phi}_t^* \tau^m_{0,s} \rVert_{C^0(U_t )} 
      \int_{U_t} \lvert \tilde{\Psi}_{t *} \phi \rvert 
+ \lVert \tilde{\Psi}_{t *} \phi - \tilde{\Phi}_t^* \tilde{\Psi}_{0 *} \phi \rVert_{C^0(U_t)} \int_{U_t} \lvert \tilde{\Phi}_t^* \tau^m_{0,s}\rvert
\nonumber \\
& \hspace*{0mm} \leq  \lVert \tau^m_{t,s} - \tilde{\Phi}_t^* \tau^m_{0,s} \rVert_{C^0(U_t )} 
      \mathrm{vol}(\mathbb{F}_{symp}) \lVert \phi \rVert_{C^0(\mathbb{F}_{symp})} 
+ \lVert \tilde{\Psi}_{t *} \phi - \tilde{\Phi}_t^* \tilde{\Psi}_{0 *} \phi \rVert_{C^0(U_t)}.\nonumber 
\end{align}
By (\ref{1}), (\ref{2}), (\ref{3}) and (\ref{4}) we finish the proof of Lemma \ref{1st-est}.
\end{proof}

Next we introduce a function $t \colon [0, \infty) \to \mathbb{R}$ so that the holomorphic sections $\sigma^m_s$ converges to delta-function sections as $s$ goes to infinity. 
\begin{lemma}\label{2nd-est}
There exists a continuous decreasing function $t \colon [0, \infty) \to \mathbb{R}$ with $t(0)=1$ and $\lim_{s \to \infty}t(s)=0$ which satisfies the following: for any $\phi \in \Gamma((L_{symp}^\mathbb{F})^*)$, there exists a constant $C_2(s, \phi)>0$ with $\lim_{s \to \infty}C_2(s, \phi)=0$ such that
$$
\left\lvert \int_{\mathbb{F}_{symp}} \langle \phi, \tilde{\Psi}_{t(s)}^* \tau^m_{t(s),s} \rangle 
- \int_{\mu_{T_{GC}}^{-1}(m) \cap V_{0,symp}} \langle \tilde{\Psi}_{0 *}\phi , \delta_{m} \rangle \, d\theta_m \right\rvert
\le C_2(s, \phi).
$$
\end{lemma}
\begin{proof}
First we estimate the term $\Vert \tilde{\Phi}_t^* \tau^m_{0,s} \Vert_{C^0(U_t^c)}$ in Lemma \ref{1st-est}. 
Due to Proposition \ref{sigma-m} $(3)$, there exists $C_3(s)>0$ such that $\lim_{s \to \infty}C_3(s)=0$ and, for any $t > 0$,
\begin{equation}\label{5} 
\Vert \tilde{\Phi}_t^* \tau^m_{0,s} \Vert_{C^0(U_t^c)} = \Vert \tau^m_{0,s} \Vert_{C^0(U_0^c)}   \le C_3(s).
\end{equation}

Next we estimate other terms in Lemma \ref{1st-est}. 
In Subsection \ref{est-gh} we fixed an open set $B \subset \mathrm{Int}\Delta_{GC}$ such that $\mathrm{Int}\Delta_{GC} \cap (\mathfrak{t}_{GC})^*_{\mathbb{Z}} \subset B$ and $\overline{B} \subset \mathrm{Int}\Delta_{GC}$.
We set $U_0 = \mu_{T_{GC}}^{-1}(B) \cap V_{0,symp} \subset V_{0,symp}^\circ$ and $U_t=\Phi_t^{-1}(U_0)$.
Now we also take an open set $B_1 \subset \mathrm{Int}\Delta_{GC}$ such that $\mathrm{Int}\Delta_{GC} \cap (\mathfrak{t}_{GC})^*_{\mathbb{Z}} \subset B_1$ and $\overline{B_1} \subset B$.
Then, due to Proposition \ref{sigma-m} $(3)$, there exists $C_4(s)>0$ such that $\lim_{s \to \infty}C_4(s)=0$ and, for any $s \ge 0$ and $m \in \mathrm{Int}\Delta_{GC} \cap (\mathfrak{t}_{GC})^*_{\mathbb{Z}}$,   
\begin{align}\label{C2s}
\Vert \frac{\tilde{\chi}_s^* \sigma^{\tilde{m}}}
      {\Vert \tilde{\chi}_s^*\sigma^{\tilde{m}} |_{V_{0,symp}} \Vert_{L^1(V_{0,symp})}} \Vert_{C^0(\mathbb{P}_{symp} \setminus \mu_{T_{GC}}^{-1}(B_1))} \le C_4(s). 
\end{align}

Since $\rho_{0,s}=\rho_{0,0} \colon V_{0,sump} \to \mathbb{P}_{symp}$ for $s \ge 0$ by Proposition \ref{toric-subm} $(2)$, we have 
$$\rho_{0,s}(U_t^c) = \rho_{0,0}(U_t^c) \subset \mathbb{P}_{symp} \setminus \mu_{T_{GC}}^{-1}(B_1).$$
Note that 
$$
\lim_{t \to 0}\Vert \tilde{\underline{\chi}}_{t,s}^* (\tilde{\rho}_{t,comp}^* \sigma^{\tilde{m}}) \Vert_{L^1(V_{t,symp})}=\Vert \tilde{\underline{\chi}}_{0,s}^* (\tilde{\rho}_{0,comp}^* \sigma^{\tilde{m}}) \Vert_{L^1(V_{0,symp})} \ne 0.
$$
Since $U_t^c$ is compact, for each $n=1,2,\dots$, there exists $t_n \in(0,1]$ which is independent of $\phi$ and satisfies the following (\ref{7}) holds for each $s \in [n,n+1]$ and $t \in [0,t_n]$;
\begin{align}\label{7}
& \rho_{t,s}(U_t^c) \subset \mathbb{P}_{symp} \setminus \mu_{T_{GC}}^{-1}(B_1), ~~
\frac{\Vert \tilde{\underline{\chi}}_{t,s}^* (\tilde{\rho}_{t,comp}^* \sigma^{\tilde{m}}) \Vert_{L^1(V_{t,symp})}}{\Vert \tilde{\underline{\chi}}_{0,s}^* (\tilde{\rho}_{0,comp}^* \sigma^{\tilde{m}}) \Vert_{L^1(V_{0,symp})}} \ge \frac{1}{2}.
\end{align}
By (\ref{C2s}) and (\ref{7}) we have, for each $s \in [n,n+1]$ and $t \in [0,t_n]$,  
\begin{align}\label{6}
C_4(s) & \ge \Vert \frac{\tilde{\rho}_{t,s}^* (\tilde{\chi}_s^* \sigma^{\tilde{m}})}{\Vert \tilde{\underline{\rho}}_{0,s}^* (\tilde{\chi}_{s}^* \sigma^{\tilde{m}}) \Vert_{L^1(V_{0,symp})}} \Vert_{C^0(U_t^c)} \\ 
& = \Vert \frac{\tilde{\underline{\chi}}_{t,s}^* (\tilde{\rho}_{t,comp}^* \sigma^{\tilde{m}})}{\Vert \tilde{\underline{\chi}}_{0,s}^* (\tilde{\rho}_{0,comp}^* \sigma^{\tilde{m}}) \Vert_{L^1(V_{0,symp})}} \Vert_{C^0(U_t^c)} \nonumber \\
& =\frac{\Vert \tilde{\underline{\chi}}_{t,s}^* (\tilde{\rho}_{t,comp}^* \sigma^{\tilde{m}}) \Vert_{L^1(V_{t,symp})}}{\Vert \tilde{\underline{\chi}}_{0,s}^* (\tilde{\rho}_{0,comp}^* \sigma^{\tilde{m}}) \Vert_{L^1(V_{0,symp})}} \Vert \tau^m_{t,s} \Vert_{C^0(U_t^c)}. \nonumber 
\end{align}
By (\ref{7}) and (\ref{6}) we have, for each $s \in [n,n+1]$ and $t \in [0,t_n]$,  
\begin{equation}\label{10}
\Vert \tau^m_{t,s} \Vert_{C^0(U_t^c)} \le 2 C_4(s).
\end{equation}

Moreover, due to Lemmas \ref{phi-approx} and \ref{psi-approx}, taking smaller $t_n>0$ if necessary, we may also conclude that the following (\ref{8}) and (\ref{9}) hold for each $s \in [n,n+1]$ and $t \in [0,t_n]$; 
\begin{align}
& \lVert \tau^m_{t,s} - \tilde{\Phi}_{t}^* \tau^m_{0,s} \rVert_{C^0(U_{t})} \le \frac{1}{n+2} \hspace*{3mm}\text{for any $m \in \mathrm{Int}\Delta \cap (\mathfrak{t}_{GC})^*_{\mathbb{Z}}$}, \label{8} \\
& \lVert \tilde{\Psi}_{t *} \phi - \tilde{\Phi}_t^* \tilde{\Psi}_{0 *} \phi \rVert_{C^0(U_t)} \le \frac{\lVert \phi \rVert_{C^1(\mathbb{F}_{symp})}}{n+2} \hspace*{3mm}\text{for any $\phi \in \Gamma((L_{symp}^{\mathbb{F}})^*)$}.\label{9}
\end{align}

By Lemma \ref{1st-est} together with (\ref{5}), (\ref{10}), (\ref{8}) and (\ref{9}) we have, for each section $\phi \in \Gamma((L_{symp}^{\mathbb{F}})^*)$, $n=1,2, \dots$, $s \in [n,n+1]$ and $t \in [0,t_n]$
\begin{align*}
& \left\lvert \int_{\mathbb{F}_{symp}} \langle \phi, \tilde{\Psi}_t^*\tau^m_{t,s} \rangle  
- \int_{\mu_{T_{GC}}^{-1}(m) \cap V_{0,symp}} \langle \tilde{\Psi}_{0 *}\phi , \delta_{m} \rangle \, d\theta_m \right\rvert\\
& \le C_1(s, \phi) + \mathrm{vol}(\mathbb{F}_{symp}) \Vert \phi \Vert_{C^0(\mathbb{F}_{symp})} (2 C_4(s) + C_3(s)) \\
& \hspace*{3mm} + \mathrm{vol}(\mathbb{F}_{symp}) \Vert \phi \Vert_{C^0(\mathbb{F}_{symp})} \frac{1}{s+1} + \frac{\lVert \phi \rVert_{C^1(\mathbb{F}_{symp})}}{s+1}.  
\end{align*}
 
We can take a continuous decreasing function $t \colon [0, \infty) \to \mathbb{R}$ with $t(0)=1$ and $\lim_{s \to \infty}t(s)=0$ such that $t(n) \le t_n$ for $n >> 0$.
Thus we finish the proof of Lemma \ref{2nd-est}. 
\end{proof}

We use $t(s)$ in Lemma \ref{2nd-est} to define the complex structure $ J_s $ by (\ref{js}) and the holomorphic section $\sigma^m_s$ by (\ref{sigma-ms}).
If we recall $\tilde{\Psi}_{t(s)}^* \tau^m_{t(s),s}= \frac{\sigma^m_s}{\Vert \sigma^m_s \Vert_{L^1(\mathbb{F}_{symp})}}$, then we have
$$ 
\lim_{s \to \infty} \int_{\mathbb{F}_{symp}} \left\langle \phi, \frac{\sigma^m_s}{\Vert \sigma^m_s \Vert_{L^1(\mathbb{F}_{symp})}} \right\rangle 
= \int_{\mu_{T_{GC}}^{-1}(m) \cap V_{0,symp}} \langle \tilde{\Psi}_{0 *}\phi , \delta_{m} \rangle \, d\theta_m .
$$
Due to Corollary \ref{nnu2}, if we define a covariantly constant section $\delta_m^{\mathbb{F}}$ of $(L^{\mathbb{F}}, h^{\mathbb{F}}, \nabla^{\mathbb{F}})|_{\mu_{GC}^{-1}(m)}$ by pulling 
$\delta_{m}$ on $\mu_{T_{GC}}^{-1}(m) \cap V_{0,symp}$ back by $\tilde{\Psi}_0$, then we have the desired convergence in Theorem~\ref{mainth} (4).
\vspace{5mm}
\noindent 
Department of Mathematics and Computer Science \\
Mount Allison University\\
67 York St, Sackville, NB, E4L 1E6, Canada \\
mhamilton@mta.ca 
\vspace{5mm}\\
Graduate School of Mathematical Sciences \\
University of Tokyo \\
3-8-1 Komaba, Meguro-ku, Tokyo, 153-8914, Japan \\
konno@ms.u-tokyo.ac.jp
\end{document}